\documentclass[11pt]{article}

\usepackage[top=2.5cm, bottom=2.5cm, left=2.5cm, right=2.5cm]{geometry}
\usepackage{amsfonts}
\usepackage{latexsym}
\usepackage{amsmath}
\usepackage{setspace}
\usepackage{comment}
\usepackage{bm}
\usepackage{graphicx}
\usepackage{subcaption}
\usepackage{epstopdf}
\usepackage{color}
\usepackage{algorithmicx}
\usepackage{algorithm}
\usepackage{algpseudocode}
\usepackage{amsopn}
\usepackage{amssymb}
\usepackage{mathtools}
\usepackage{amsthm}
\usepackage{booktabs}
\usepackage{placeins}
\usepackage{siunitx}

\newtheorem{definition}{Definition}[section]
\newtheorem{lemma}{Lemma}[section]
\newtheorem{assumption}{Assumption}[section]
\newtheorem{theorem}{Theorem}[section]
\newtheorem{proposition}{Proposition}[section]

\newcommand{\mini}{\mathop{\mbox{minimize}}}
\newcommand{\maxi}{\mathop{\mbox{maximize}}}
\newcommand{\subj}{\mbox{subject to}}
\newcommand{\R}{\mathbb{R}}
\newcommand{\M}{\mathcal{M}}
\newcommand{\rmD}{\mathrm{D}}
\newcommand{\grad}{\mathrm{grad}}
\newcommand{\Hess}{\mathrm{Hess}}

\title{A stabilized sequential quadratic programming method \\ for degenerate nonlinear optimization problems \\ on Riemannian manifolds}
\author{Yuya Yamakawa\thanks{Graduate School of Management, Tokyo Metropolitan University, 1-1 Minami-Osawa, Hachioji-shi, Tokyo 192-0397, Japan (yuya@tmu.ac.jp)}, Mamoru Oka\thanks{Graduate School of Informatics, Kyoto University, Yoshida-Honmachi, Sakyo-ku, Kyoto 606-8501, Japan (oka\_0319\_2000@docomo.ne.jp)}}

\begin{document}

\maketitle

\begin{abstract}
We propose a stabilized sequential quadratic programming (SQP) method for degenerate constrained optimization problems on Riemannian manifolds. The problem considered in this study is a Riemannian nonlinear programming problem (RNLP) with equality and inequality constraints, where classical constraint qualifications may fail. While existing Riemannian SQP methods guarantee global convergence only under constraint qualifications, their convergence behavior is not ensured for degenerate problems. To address this limitation, we extend the stabilized SQP framework from Euclidean spaces to Riemannian manifolds. Without assuming any constraint qualification, we prove that the generated sequence has an accumulation point that is a Karush--Kuhn--Tucker (KKT) point, an approximate KKT (AKKT) point, or a stationary point of an associated feasibility problem. Finally, we conduct numerical experiments to confirm the effectiveness of the proposed method for degenerate problems.
\end{abstract}

{\bf keywords:} Riemannian stabilized SQP method, Riemannian optimization, degenerate problem, Riemannian manifold

\section{Introduction}
In this paper, we consider the following optimization problem:
\begin{align}\label{pro:RNLO}
\begin{aligned}
& \mini_{x \in \M} && f(x)
\\
& \subj && g_{i}(x) = 0 ~~(i=1, \ldots, m)
\\
& && h_{j}(x) \geq 0 ~~(j=1, \ldots, \ell)
\end{aligned}
\end{align}
where $\M$ is a $d$-dimensional connected and complete Riemannian manifold, the functions $f \colon \M \to \R$, $g_{i} \colon \M \to \R ~ (i=1, \ldots, m)$, and $h_{j} \colon \M \to \R ~ (j=1,\ldots,\ell)$ are twice continuously differentiable on $\M$, respectively. Moreover, we define the functions $g$ and $h$ as $g(x) = [ g_{1}(x), \ldots, g_{m}(x)]^{\top}$ and $h(x) = [h_{1}(x), \ldots, h_{\ell}(x)]^{\top}$. Problem~\eqref{pro:RNLO} is generally called a Riemannian nonlinear programming problem~(RNLP) and is a natural extension of a nonlinear programming problem~(NLP) in Euclidean spaces. Indeed, problem~\eqref{pro:RNLO} becomes NLP in the case that $\M = \R^{d}$.
\par
RNLPs arise in various fields, such as machine learning and control theory, and have been studied and developed since the 2000s, particularly regarding methods for solving the unconstrained version of RNLPs. Most of these methods are extensions of existing techniques for optimization problems in Euclidean spaces to those in Riemannian manifolds, such as Newton methods~\cite{Absil:2008}, trust-region methods~\cite{Absil:2008}, and conjugate gradient methods~\cite{Sato:2015}. In recent years, RNLPs have been actively studied, and several optimization methods have already been proposed, including Riemannian augmented Lagrangian methods~\cite{Liu:2020}, exact penalty methods~\cite{Liu:2020}, Riemannian sequential quadratic programming~(RSQP) methods~\cite{Obara:2022}, and primal-dual interior point methods~\cite{Lai:2024}.
\par
The purpose of most optimization methods is to find a KKT point that satisfies the Karush-Kuhn-Tucker~(KKT) conditions. The KKT conditions are the first-order optimality conditions, and it is generally known that the KKT conditions serve as first-order necessary optimality conditions only under an appropriate constraint qualification~(CQ), such as the Mangasarian-Fromovitz CQ~(MFCQ), the Robinson CQ~(RCQ), and the linear independence CQ~(LICQ). However, for degenerate problems for which standard CQs fail, there is a possibility that their local optima do not satisfy the KKT conditions, and it is difficult for ordinary optimization methods to find their KKT points. To overcome such a difficulty, the Approximate Gradient Projection~(AGP) conditions~\cite{Martinez:2003} and Approximate KKT~(AKKT) conditions~\cite{Andreani:2011} have been proposed for NLPs since the 2000s. They are satisfied at all local optima regardless of whether CQs hold or not and are often called sequential optimality conditions because some sequences are used in their definitions. In recent years, the concept of sequential optimality conditions has been extended to nonlinear second-order cone programming~\cite{Andreani:2024}, nonlinear semidefinite programming~(NSDP)~\cite{Andreani:2020}, RNLP~\cite{Yamakawa:2022}, and optimization problems in function spaces~\cite{Kanzow:2018}.
\par
Sequential quadratic programming~(SQP) methods are known as powerful optimization methods for solving nonlinear optimization problems. They iteratively solve quadratic programming problems obtained by approximating the objective quadratically and the constraint functions linearly, and several types of them have been proposed so far~\cite{Byrd:2008,Izmailov:2010,Wright:1998}. Moreover, SQP methods have been extended to NSDPs~\cite{Correa:2004,Fares:2002} and RNLPs~\cite{Obara:2022}.
In particular, Wright~\cite{Wright:1998} proposed a stabilized SQP method based on the classical SQP method. This method was proposed for solving degenerate problems and it iteratively solves the stabilized quadratic subproblems designed for degenerate problems. Inspired by Wright~\cite{Wright:1998}, many researchers have proposed several stabilized SQP methods, such as Hager~\cite{Hager:1999}, Gill and Robinson~\cite{Gill:2013}, Izmailov et al.~\cite{Izmailov:2015}, and Gill et al.~\cite{Gill:2017}. Moreover, stabilized SQP methods have been extended to nonlinear semidefinite programming~\cite{Yamakawa:2022-2} and optimization problems in Banach spaces~\cite{Yamakawa:2023,Yamakawa:2026}, but have not yet been proposed for RNLPs. 
\par
In this study, we propose a stabilized SQP method for problem~\eqref{pro:RNLO} and establish global convergence of the proposed method under some assumptions. As stated above, stabilized SQP methods for the optimization problem on Riemannian manifolds have not been studied whereas the ordinary SQP method has been extended to RNLPs by Obara et al.~\cite{Obara:2022}, and hence we believe that the current research is the first of this kind. Moreover, the proposed method is designed to solve degenerate problems and is guaranteed to converge globally to a KKT point, an AKKT point, or a stationary point of an associated feasibility problem without assuming any CQs. Meanwhile, such a convergence property cannot be seen in the RSQP methods proposed by Obara et al.~\cite{Obara:2022}. In the numerical experiments, we confirm that the proposed method is more effective than the existing RSQP method for degenerate problems on Riemannian manifolds.
\par
This paper is organized as follows. In Section~\ref{sec:Preliminaries}, we introduce some terminology related to optimality conditions and CQs. Section~\ref{sec:SQP_method} outlines an ordinary SQP method and its variants. In Section~\ref{sec:SSQP}, we propose a stabilized SQP method on a Riemannian manifold and prove its global convergence. Section~\ref{sec:Numerical_experiments} reports numerical experiments and confirms the effectiveness of the proposed method for degenerate problems.
Conclusions and future work are stated in Section~\ref{sec:conclusion}.

\section{Preliminaries}\label{sec:Preliminaries}
Throughout this paper, we use the following notation. The set of positive integers is denoted by ${\mathbb{N}}$. Let $p \in \mathbb{N}$ be an integer. The $p$-dimensional Euclidean space is denoted by $\R^{p}$. We define $\R^{p}_{+} \coloneqq \{x \in \R^{p} ; x \geq 0\}$. The tangent space of $\M$ at $x$ is defined by $T_{x}{\M}$. The tangent bundle of $\M$ is denoted by $T \M$. The retraction on $\M$ is written as $R \colon T \M \to \M$. In particular, we denote by $R_{x}$ the restriction of $R$ to $T_{x} \M$. The Riemannian metric is denoted by $\{ \langle \cdot, \cdot \rangle_{x} \}_{x \in \M}$, where $\langle \cdot, \cdot \rangle_{x} \colon T_{x} \M \times T_{x} \M \rightarrow \R$. The Riemannian metric induces the norm $\Vert \xi \Vert_{x} \coloneqq \sqrt{\langle \xi, \xi \rangle_{x}}$ for $\xi\in T_{x}\M$. Let $a \in \R^{p}$, $b \in \R^{p}$, and $v \in \R^{p}$ be vectors. We define $\langle a, b \rangle \coloneqq a^{\top} b$ and $\| v \| \coloneqq \sqrt{\langle v, v \rangle}$. If $a$ and $b$ are scalars, $\mathrm{max}(a,b)$ represents the larger one of $a$ and $b$. The scalar $[v]_i$ represents the $i$-th element of $v$. 
We denote $[v]_{+} \coloneqq [\max([v]_1,0), \ldots, \max([v]_p, 0)]^{\top}$. The Riemannian distance on ${\M}$ is denoted by $d \colon \M \times \M \rightarrow \R$. Throughout the paper, we assume that $(\M, d)$ is a complete metric space. For any $v \in \M$ and $\eta > 0$, we denote $B(v, \eta) \coloneqq \{ x \in \M; d(x, v) \leq \eta \}$.
Let $x \in \M$, and let $\phi \colon \M \to \R$ and $\Phi \colon \M\rightarrow \R^{p}$ be functions. If $\phi$ has a Riemannian gradient at $x$, it is denoted by $\grad \phi(x)$. If $\Phi$ has a derivative at $x$, it is represented by $\rmD \Phi(x)$. Let $\psi \colon \M \times \R^{p} \to \R \colon (x, y) \mapsto \psi(x,y)$ be a function. 
Let $\nabla$ be the Levi-Civita connection on $\M$. 
If $\phi$ has a Riemannian Hessian at $x$, it is denoted by $\Hess \phi(x)$, that is, $\Hess \phi(x)[\xi] = \nabla_{\xi} \grad \phi(x)$ for all $\xi \in T_{x} \M$. If $\M$ is the Euclidean space, the symbol $\nabla$ denotes the Euclidean gradient for real-valued functions, namely, we represent $\grad \phi(x)$ as $\nabla \phi(x)$. For a closed convex set $C \subset \R^{p}$, we denote by $P_{C}(x)$ the metric projection of $x \in \R^{p}$ onto $C$. For any $S \subset \R^{p}$, the interior of $S$ is denoted by $\mathrm{int}(S)$.
\par
Now, we will introduce several optimality conditions for problem~\eqref{pro:RNLO}. To this end, the Lagrange function is defined as follows:
\begin{align*}
L(x,y,z) \coloneqq f(x) - y^{\top} g(x) - z^{\top} h(x),
\end{align*}
where $y \in \R^{m}$ and $z \in \R^{\ell}$ are Lagrange multipliers for $g(x) = 0$ and $h(x) \geq 0$, respectively. Note that the Riemannian gradient and the Riemannian Hessian of $L$ with respect to $x$ are denoted by $\grad_{x} L(x,y,z)$ and $\Hess_{x} L(x,y,z)$, respectively.

\begin{definition}[{\cite[Definition 2.3]{Liu:2020}}]
The KKT conditions for problem~\eqref{pro:RNLO} hold at $x \in \M$ if there exist $y\in \R^{m}$ and $z \in \R^{\ell}$ such that      
\begin{gather*}
\grad_{x} L(x, y, z) = 0, \quad g(x)=0, \quad h(x)\geq 0, \quad z\geq 0, \quad \langle z, h(x) \rangle = 0.
\end{gather*}
\end{definition}
A point $x \in \M$ satisfying the KKT conditions is called a KKT point. The KKT conditions are well known as the first-order optimality conditions. However, for degenerate problems, there is a possibility that their local optima do not satisfy the KKT conditions. In other words, the KKT conditions provide the necessary optimality for~\eqref{pro:RNLO} under some CQ.
\par
In Riemannian optimization, several CQs have been extended. For example, the definitions of the linear independence CQ~(LICQ) and the Mangasarian-Fromovitz CQ~(MFCQ) are given by Bergmann and Herzog~\cite{Bergmann:2019}, and the definition of the Robinson~CQ~(RCQ) is provided by Yamakawa and Sato~\cite{Yamakawa:2022}. In this study, we introduce the definition of the RCQ.

\begin{definition}[{~\cite[Definition 4]{Yamakawa:2022}}]
We say that a feasible point~$x \in \M$ satisfies the RCQ if the following condition holds:
\begin{align} \label{cond:ERCQ}
0 \in \mathrm{int}\left(
\begin{bmatrix}
g(x)
\\
h(x)
\end{bmatrix}
+
\begin{bmatrix}
\rmD g(x)
\\
\rmD h(x)
\end{bmatrix}
T_{x} \M
-
\begin{bmatrix}
\{ 0 \}
\\
\R^{\ell}_{+}
\end{bmatrix}
\right),
\end{align}
where 
\begin{align*}
\begin{bmatrix}
\rmD g(x)
\\
\rmD h(x)
\end{bmatrix}
T_{x} \M
\coloneqq \left\{ 
\begin{bmatrix}
\rmD g(x)\xi
\\
\rmD h(x)\xi
\end{bmatrix}
\in \R^{m + \ell}; \xi \in T_{x} \M \right\}.
\end{align*}
Moreover, if a point $x \in \M$, which is not necessarily feasible for~\eqref{pro:RNLO}, meets condition~\eqref{cond:ERCQ}, then we say that the extended RCQ (ERCQ) holds at $x$.
\end{definition}
\par
Recently, the concept of the approximate KKT~(AKKT) conditions was proposed as the optimality condition for degenerate optimization problems. Yamakawa and Sato proposed the AKKT conditions for Riemannian optimization.   
\begin{definition} {\rm \cite[Definition 5]{Yamakawa:2022}}
The AKKT conditions for problem~\eqref{pro:RNLO} hold at $x \in \M$ if $g(x) = 0$, $h(x) \geq 0$, and there exist $\{ x_k \} \subset \M$, $\{ y_{k} \} \subset \R^{m}$, and $\{ z_{k} \}\subset \R^{\ell}_{+}$ such that
\begin{gather*}
\lim_{k \to \infty} x_{k} = x, \quad \lim_{k \to \infty} \| \grad_{x} L(x_{k}, y_{k}, z_{k}) \|_{x_k} = 0, \quad \lim_{k \to \infty} \langle z_k, [h(x_k)]_{+} \rangle = 0.
\end{gather*}
\end{definition}
We call $x \in \M$ an AKKT point if the AKKT conditions are satisfied at $x$. Moreover, we call $\{ ( x_{k}, y_{k}, z_{k})\} \subset \M \times \R^{m} \times \R_{+}^{\ell}$ that are used for defining an AKKT point $x$ an AKKT sequence corresponding to $x$. The AKKT conditions are satisfied at all local optima regardless of whether CQs hold or not, where we notice that this property is satisfied under the completeness of $\M$. For details, see~\cite[Theorem~1]{Yamakawa:2022}. They are often called sequential optimality conditions because some sequences are used in their definitions.

\section{SQP-type methods}\label{sec:SQP_method}
This section provides a brief outline of the SQP-type method for the NLP and RNLP. We first introduce the ordinary SQP methods for the following NLP~\eqref{pro:NLP}, which is obtained by $\M = \mathbb{R}^{n}$.
\begin{align}
\begin{aligned}\label{pro:NLP}
& \mini_{x \in \R^{n}} && f(x)
\\
& \subj && g(x) = 0,
\\
& && h(x) \geq 0. 
\end{aligned}
\end{align}
In general, the ordinary SQP methods iterate the following three steps.
\begin{description}
\item[{\rm Step 1:}] Obtain a search direction and new Lagrange multipliers by solving a quadratic subproblem.

\item[{\rm Step 2:}] Choose a step size based on a certain condition, such as the Armijo condition.

\item[{\rm Step 3:}] Update the main iterate and parameters.
\end{description}
Let $k$ denote the iteration number. In the first step of the ordinary SQP methods, the following subproblem~\eqref{pro:Euc_SQP} is solved.
\begin{align}
\begin{aligned}\label{pro:Euc_SQP}
& \mini_{\xi \in \R^{n}} && \langle \nabla f(x_{k}), \xi \rangle + \frac{1}{2} \langle H_{k} \xi, \xi \rangle
\\
& \subj && g(x_{k}) + \nabla g(x_{k})^{\top} \xi = 0,
\\
& && h(x_{k}) + \nabla h(x_{k})^{\top} \xi \geq 0,
\end{aligned}  
\end{align}
where $H_{k}$ represents $\nabla_{xx}^{2} L(x_k,y_k,z_k)$ or its approximation. This problem is derived from the second-order approximation regarding the min-max problem of the following Lagrange function. In the following, we briefly explain how to derive subproblem~\eqref{pro:Euc_SQP}. Note that NLP~\eqref{pro:NLP} is equivalent to 
\begin{align}
\mini_{x \in \R^{n}} ~ \maxi_{(y,z) \in \R^{m} \times \R_{+}^{\ell}} ~ L(x,y,z). 
\end{align}
Now, we consider the second-order approximation of $L$ with respect to $x$, that is,
\begin{align}\label{pro:second-order}
\begin{aligned}
& \mini_{\xi \in \R^{n}} ~ \maxi_{(\zeta, \eta) \in \R^{m} \times \R_{+}^{\ell}} ~ \langle \nabla f(x_{k}), \xi \rangle + \frac{1}{2} \langle \nabla_{xx}^{2} L(x_{k}, y_{k}, z_{k}) \xi, \xi \rangle 
\\
& \hspace{35mm} - \zeta^{\top}(g(x_{k}) + \nabla g(x_{k})^{\top} \xi) - \eta^{\top}(h(x_{k}) + \nabla h(x_{k})^{\top} \xi). 
\end{aligned}
\end{align}
Then, we can easily see that~\eqref{pro:second-order} is equivalent to~\eqref{pro:Euc_SQP} if $H_k=\nabla_{xx}^2L(x_k,y_k,z_k)$. Namely, the subproblem of the ordinary SQP methods is derived from the second-order approximation of min-max problem of the Lagrange function. The ordinary SQP methods find a KKT point $(\xi^{\ast}, \zeta^{\ast}, \eta^{\ast})$, set a search direction by $p_{k} \coloneqq \xi^{\ast}$, and update the Lagrange multipliers $(y_{k}, z_{k})$ by $(y_{k+1}, z_{k+1}) \coloneqq (\zeta^{\ast}, \eta^{\ast})$. 
\par
In the second step, the step size is determined. In particular, the ordinary SQP methods often use Armijo's line search with a merit function to establish the global convergence. 
Although several merit functions have been proposed, the following merit function is one of the most well-known.
\begin{align} \label{merit_well-known}
\phi_{\rho}(x) \coloneqq f(x) + \rho \left( \sum_{i=1}^{m} |g_i(x)| + \sum_{j=1}^{\ell} \max \{ -h_{j}(x), 0 \} \right),
\end{align}
where $\rho > 0$ is a penalty parameter.     
\par
In the third step, by using the step size $\alpha_{k} > 0$, the main iterate is updated as $x_{k+1} \coloneqq x_k + \alpha_{k} p_{k}$. Moreover, the parameters are updated appropriately, and the method returns to the first step.
\par
The ordinary SQP methods have various variants, such as inexact SQP methods~\cite{Byrd:2008}, truncated SQP methods~\cite{Izmailov:2010}, and stabilized SQP~(SSQP) methods~\cite{Gill:2013,Wright:1998}. In particular, the SSQP methods have been actively studied for degenerate problems. The SSQP methods adopt different subproblems from the ordinary SQP methods, whereas they iterate the same three steps described above. The subproblem of the SSQP methods is given as follows: 
\begin{align}
\begin{aligned} \label{pro:Euc_SSQP}
& \mini_{(\xi,\zeta,\eta) \in V} && \langle \nabla f(x_{k}), \xi \rangle + \frac{1}{2} \langle H_{k} \xi, \xi \rangle + \frac{\sigma_{k}}{2} \Vert \zeta \Vert^{2} + \frac{\sigma_{k}}{2} \Vert \eta \Vert^{2}
\\
& \subj && g(x_{k}) + \nabla g(x_{k})^{\top} \xi + \sigma_{k} (\zeta - y_{k}) = 0,
\\
&&& h(x_{k}) + \nabla h(x_{k})^{\top} \xi + \sigma_{k} (\eta - z_{k}) \geq 0.
\end{aligned}
\end{align}
where $V \coloneqq \R^{n} \times \R^{m} \times \R^{\ell}$ and $\sigma_{k} > 0$ is a penalty parameter. If $H_k = \nabla_{xx}^{2} L(x_k,y_k,z_k)$, this is equivalent to the following problem, which is obtained by adding the proximal terms $\frac{\sigma_{k}}{2} \Vert \zeta - y_{k} \Vert^{2}$ and $\frac{\sigma_{k}}{2} \Vert \eta - z_{k} \Vert^{2}$ to problem~\eqref{pro:second-order}.
\begin{align}
\begin{aligned}
& \mini_{\xi \in \R^{n}} ~ \maxi_{(\zeta, \eta) \in \R^{m} \times \R_{+}^{\ell}} ~ \, \langle \nabla f(x_{k}), \xi \rangle + \frac{1}{2} \langle \nabla_{xx}^{2} L(x_{k}, y_{k}, z_{k}) \xi, \xi \rangle
\\
& \hspace{45mm} - \zeta^{\top}(g(x_{k}) + \nabla g(x_{k})^{\top} \xi) - \frac{\sigma_{k}}{2} \Vert \zeta - y_{k} \Vert^{2}
\\
& \hspace{55mm} - \eta^{\top}(h(x_{k}) + \nabla h(x_{k})^{\top} \xi) - \frac{\sigma_{k}}{2} \Vert \eta - z_{k} \Vert^{2}. 
\end{aligned}
\end{align}
The proximal terms $\frac{\sigma_{k}}{2} \Vert \zeta - y_{k} \Vert^{2}$ and $\frac{\sigma_{k}}{2} \Vert \eta - z_{k} \Vert^{2}$ help prevent the Lagrange multipliers from becoming unbounded, and hence it is generally known that they work well for degenerate problems.


\par
Recently, Obara et al.~\cite{Obara:2022} have proposed an SQP method, which is called a Riemannian SQP~(RSQP) method, for RNLPs. This method is a natural extension of the ordinary SQP methods for NLPs and iteratively solves the following subproblem at the $k$-th iteration. 
\begin{align}
\begin{aligned}\label{pro:Rei_SQP}
& \mini_{\xi \in T_{x_{k}} \M} && \langle \grad f(x_{k}), \xi \rangle_{x_{k}} + \frac{1}{2} \langle H_{k}[\xi], \xi \rangle_{x_{k}}
\\
& \subj && g(x_{k}) + \rmD g(x_{k}) [\xi] = 0,
\\
& && h(x_{k}) + \rmD h(x_{k}) [\xi] \geq 0,
\end{aligned}
\end{align}
where $H_{k}$ denotes $\Hess_{x} L(x_{k}, y_{k}, z_{k})$ or its approximation. Moreover, the following merit function~$P_{\rho_k}(x)$ is used in Armijo's line-search, which is proposed by Liu and Boumal~\cite{Liu:2020}:
\begin{align*}
P_{\rho_k}(x) \coloneqq f(x) + \rho_k \left( \sum_{i=1}^{m} |g_i(x)| + \sum_{i=1}^{\ell} \max (0, -h_i(x)) \right), 
\end{align*}
where $\rho_k$ is a penalty parameter. This function is also a natural extension of~\eqref{merit_well-known}. Obara et al.~\cite{Obara:2022} showed global convergence and local fast convergence of the RSQP method.
\par
In this study, we propose a stabilized Riemannian SQP~(RSSQP) method for problem~\eqref{pro:RNLO}. The subproblem of the RSSQP method~\eqref{pro:RNLOpar} is a natural extension of the SSQP methods for NLPs, and thus it is provided as follows:
\begin{align}
\begin{aligned}\label{pro:RNLOpar}
& \mini_{(\xi, \zeta, \eta) \in {\cal V}_{k}} && \langle \grad f(x_{k}), \xi \rangle_{x_{k}} + \frac{1}{2} \langle H_{k}[\xi], \xi \rangle_{x_{k}} + \frac{\sigma_{k}}{2} \Vert \zeta \Vert^{2} + \frac{\sigma_{k}}{2} \Vert \eta \Vert^{2}
\\
& \subj && g(x_{k}) + \rmD g(x_{k}) [\xi] + \sigma_{k}( \zeta - \widehat{y}_{k} )=0,
\\
&&& h(x_{k}) + \rmD h(x_{k}) [\xi] + \sigma_{k} ( \eta - \widehat{z}_{k} ) \geq 0,
\end{aligned}
\end{align}
where ${\cal V}_{k} \coloneqq T_{x_{k}} \M \times \R^{m} \times \R^{\ell}$, $\sigma_{k} >0$ is a penalty parameter, and $\widehat{y}_k \in \R^m$ and $\widehat{z}_k \in \R^{\ell}$ are auxiliary Lagrange multipliers. By removing equality constraints, we can reformulate subproblem~{\eqref{pro:RNLOpar}} into the following problem:
\begin{align}
\begin{aligned}\label{pro:RNLOparch}
& \mini_{(\xi, \eta) \in {\cal W}_{k}} && \langle \grad f(x_{k}) - \rmD g(x_{k})^{\ast}[s_{k}], \xi \rangle_{x_{k}} + \frac{1}{2} \langle M_{k}[\xi], \xi \rangle_{x_{k}} + \frac{\sigma_{k}}{2} \Vert \eta \Vert^{2}
\\
& \subj && \rmD h(x_{k})[\xi] + \sigma_{k} ( \eta - t_{k} ) \geq 0,
\end{aligned}
\end{align}
where ${\cal W}_{k} \coloneqq T_{x_{k}} \M \times \R^{\ell}$, and $s_{k}$, $t_{k}$, and $M_{k}$ are defined as follows:
\begin{align*}
s_{k} \coloneqq \widehat{y}_{k} - \frac{1}{\sigma_{k}} g(x_{k}), \quad t_{k} \coloneqq \widehat{z}_{k} - \frac{1}{\sigma_{k}}h(x_{k}), \quad M_{k} \coloneqq H_{k} + \frac{1}{\sigma_{k}} \rmD g(x_{k})^{\ast} \rmD g(x_{k}).
\end{align*}
The proposed RSSQP method iteratively finds an optimizer $(\xi_{k}^{\ast}, \eta_{k}^{\ast}) \in T_{x_{k}} \M \times \R^{\ell}$ of~\eqref{pro:RNLOparch} and sets a search direction $p_{k}$ and Lagrange multiplier estimates $(\widetilde{y}_{k+1}, \widetilde{z}_{k+1})$ as
\begin{align}
p_{k} \coloneqq \xi_k^{\ast}, \quad \widetilde{y}_{k+1} \coloneqq \widehat{y}_{k} - \frac{1}{\sigma_{k}}(g(x_{k}) + \rmD g(x_{k}) [\xi_k^{\ast}]), \quad \widetilde{z}_{k+1} \coloneqq \eta_k^{\ast}, \label{newpoints}
\end{align}
respectively. This study focuses on global convergence of the proposed RSSQP method, and hence, we also introduce a merit function used in Armijo's line-search. The proposed merit function is given as
\begin{align}
F(x;y,z,\sigma) \coloneqq f(x) + \frac{1}{2\sigma} \Vert \sigma y - g(x) \Vert^{2} + \frac{1}{2\sigma} \Vert [\sigma z -h(x)]_{+} \Vert^{2},
\end{align}
where $\sigma > 0$ is a penalty parameter. Note that the above function is equivalent to the augmented Lagrangian. The Riemannian gradient of the merit function is given by
\begin{align}
\grad F(x;y,z,\sigma) = \grad f(x) - \rmD g(x)^{\ast}\left[ y - \textstyle \frac{1}{\sigma} g(x) \right] - \rmD h(x)^{\ast} \left[ \left[ z - \textstyle\frac{1}{\sigma}h(x) \right]_{+} \right]. \label{eq:graddef}
\end{align}
In the next section, we describe a formal statement of the proposed RSSQP method and explain how to update the iterates and parameters.


\section{RSSQP method and its global convergence}\label{sec:SSQP}
This section presents an RSSQP method for problem~\eqref{pro:RNLO} and proves its global convergence under milder assumptions without any CQs. First, we define
the functions $\Phi_{\theta} \colon \M \times \R^m \times \R_+^{\ell} \to \R$, $\Psi_{\theta} \colon \M \times \R^m \times \R_+^{\ell} \to \R$, and $r \colon \M \times \R^m \times \R_+^{\ell} \to \R$ as
\begin{align*}
\Phi_{\theta}(x,y,z) &\coloneqq \Vert g(x) \Vert + \Vert [-h(x)]_{+} \Vert + \theta \Vert \grad_{x} L(x,y,z) \Vert_{x} + \theta | \langle z, h(x) \rangle |,
\\
\Psi_{\theta}(x,y,z) &\coloneqq \theta \Vert g(x) \Vert + \theta \Vert [-h(x)]_{+} \Vert + \Vert \grad_{x} L(x,y,z) \Vert_{x} + | \langle z, h(x) \rangle |,
\\
r(x,y,z) &\coloneqq \Vert g(x) \Vert + \Vert [-h(x)]_{+} \Vert + \Vert \grad_{x} L(x,y,z) \Vert_{x} + | \langle z, h(x) \rangle |,
\end{align*}
respectively, where $\theta \in (0,1)$ is a constant. Note that the KKT conditions are satisfied if and only if $\Phi_{\theta}(x,y,z) = \Psi_{\theta}(x,y,z) = r(x,y,z) = 0$. 
\par
Next, we outline the proposed RSSQP method. Although it also consists of the three steps as described in Section~\ref{sec:SQP_method}, the third step particularly represents the key features of the proposed method. Thus, we will focus on this step, that is, updating rules regarding the main iterate, the Lagrange multipliers, and the parameters. In the subsequent arguments, we denote by $x_{k}$ the $k$-th iteration point and use the symbols $p_{k}$, $\widetilde{y}_{k+1}$, and $\widetilde{z}_{k+1}$ defined in~\eqref{newpoints}.

\subsubsection*{Updating rule of main iterate}
In the ordinary SQP-type method for optimization problems in Euclidean spaces, the main iterate is updated by $x_{k+1} \coloneqq x_{k} + \alpha_{k} p_{k}$ after computing the search direction $p_{k} \in T_{x_{k}} \M$ and step size $\alpha_{k} > 0$. However, such a new point $x_{k+1}$ does not make sense for Riemannian optimization. Thus, the new point is generated by the following updating rule: $x_{k+1} \coloneqq R_{x_{k}}(\alpha_{k} p_{k})$, where $R_{x_{k}} \colon T_{x_{k}}\M \to \M$ denotes a restriction to $T_{x_{k}} \M$ of the retraction $R$. This updating rule is also seen in Obara et al.~\cite{Obara:2022} and Yamakawa et al.~\cite{Yamakawa:2022}.

\subsubsection*{Updating rules of Lagrange multipliers and parameters}
One of the distinctive features of the proposed method is shown in Algorithm~\ref{alg:sub}, and it generates two types of Lagrange multiplier sequences: the main Lagrange multiplier sequence $\{(y_k,z_k)\}$ and the auxiliary Lagrange multiplier sequence $\{(\widehat{y}_k,\widehat{z}_k)\}$. The auxiliary Lagrange multiplier sequence is generated to be bounded, and these two sequences make the algorithm more stable. Since the proposed method aims to obtain a KKT point or an AKKT point, we need to evaluate whether the generated sequence is approaching such a point or not. To this end, we consider the following three cases. 
\begin{description}
\item[{\rm Case 1:}] $x_{k+1}$ is closer to a KKT point compared to $x_k$.
\item[{\rm Case 2:}] $x_{k+1}$ is closer to a stationary point of $F$ compared to $x_k$.
\item[{\rm Case 3:}] otherwise.
\end{description}
Steps~4.1 and 4.2 in Algorithm~\ref{alg:sub} determine whether Case~1 holds or not. Therefore, if it is true, we set $(y_{k+1},z_{k+1}) \coloneqq (\widetilde{y}_{k+1}, \widetilde{z}_{k+1})$ and $({\widehat{y}}_{k+1},{\widehat{z}}_{k+1}) \coloneqq (\widetilde{y}_{k+1}, \widetilde{z}_{k+1})$ because the new point $(x_{k+1}, \widetilde{y}_{k+1}, \widetilde{z}_{k+1})$ is approaching a KKT point. Step 4.3 investigates whether Case~2 holds or not. If it is true, Algorithm~\ref{alg:sub} updates the Lagrange multipliers by updating rules based on augmented Lagrangian methods because this step corresponds to the minimization of the augmented Lagrangian $F$. In the case of Step~4.4, that is, Case~3, the Lagrange multipliers are updated with no changes. If this step occurs repeatedly, the multiplier estimates and the parameters remain unchanged, which is undesirable. However, the algorithm cannot remain in Step~4.4 at every sufficiently large iteration under appropriate assumptions.

\subsubsection*{Proposed RSSQP method}
In the remainder of this section, we provide a formal statement of the proposed RSSQP method, that is, Algorithm~\ref{alg:main}, and its convergence analysis. The main goal of the analysis is to establish global convergence of Algorithm~\ref{alg:main} under some assumptions.

\begin{algorithm}[htbp] \caption{(Proposed RSSQP method)} \label{alg:main}
\begin{algorithmic}[1]
\Require Choose constants $\beta \in (0,1)$, $\varepsilon \in (0,1)$, and $\theta \in (0,1)$. Set nonempty compact convex sets $C \subset \R^{m}$ and $D \subset \R_{+}^{\ell}$. Initialize $(x_0,y_0,z_0) \in \M \times \R^{m} \times \R^{\ell}$, $\widehat{y}_{0} \coloneqq y_0$, $\widehat{z}_0 \coloneqq z_0$, $\sigma_0 > 0$, $\phi_0 > 0$, $\psi_0 > 0$, $\gamma_0 > 0$, and $k\coloneqq 0$, where $y_0 \in C$ and $z_0 \in D$.

\Repeat
\State Select $H_k$ so that it is self-adjoint, and $\langle ( H_k + \frac{1}{\sigma_k} \rmD{g}(x_k)^{\ast} \rmD{g}(x_k))[ \, \cdot \,], \, \cdot \, \rangle_{x_k}$ is coercive. Set $M_k$, $s_k$, and $t_k$ as follows:
\begin{gather*}
\textstyle M_k \coloneqq H_k + \frac{1}{\sigma_k} \rmD g(x_k)^{\ast} \rmD g(x_k),
~~
\textstyle s_k \coloneqq \widehat{y}_{k} - \frac{1}{\sigma_k} g(x_k), ~~ t_k \coloneqq \widehat{z}_{k} - \frac{1}{\sigma_k} h(x_k).
\end{gather*}

\State Solve problem~\eqref{pro:RNLOparch} and obtain its optimizer $(\xi_{k}^{\ast}, \eta_{k}^{\ast})\in T_{x_{k}} \M \times \R_{+}^{\ell}$. \label{step:subpro}

\State Set $(p_k, \widetilde{y}_{k+1}, \widetilde{z}_{k+1})$ as follows:
\begin{align}
\textstyle p_k \coloneqq \xi_{k}^{\ast}, \quad \widetilde{y}_{k+1} \coloneqq \widehat{y}_k - \frac{1}{\sigma_k} (g(x_k) + \rmD g(x_k)[\xi_{k}^{\ast}]), \quad \widetilde{z}_{k+1} \coloneqq \eta_{k}^{\ast}.
\end{align}

\State Find a step size $\alpha_k = {\beta}^{j_k}$, where $j_k$ is the smallest nonnegative integer satisfying
\begin{align}\label{eq:armijo-rule}
F(R_{x_{k}} (\beta^{j_k} p_k); \widehat{y}_k, \widehat{z}_k, \sigma_{k}) \leq F(x_k; \widehat{y}_k, \widehat{z}_k, \sigma_{k}) + \varepsilon \beta^{j_k} \langle \grad F(x_k; \widehat{y}_k, \widehat{z}_k, \sigma_k), p_k \rangle_{x_{k}}. 
\end{align}

\State Compute $x_{k+1}=R_{x_k}(\alpha_{k} p_k)$.

\State Set $(y_{k+1}, z_{k+1}, \widehat{y}_{k+1}, \widehat{z}_{k+1}, \sigma_{k+1}, \phi_{k+1}, \psi_{k+1}, \gamma_{k+1})$ by Algorithm~\ref{alg:sub}. \label{update:Algo2}

\State Update $k \leftarrow k + 1$.

\Until $(x_k,y_k,z_k)$ satisfies a stopping criterion.
\end{algorithmic}
\end{algorithm}

\begin{algorithm}[t] \caption{(Updating rules of Lagrange multipliers and parameters)}
\label{alg:sub}
\begin{algorithmic}[1]
\Require Select a constant $\theta \in (0, 1)$ and nonempty compact convex sets $C \subset \R^{m}$ and $D \subset \R_{+}^{\ell}$. Input $x_{k+1}, \widetilde{y}_{k+1}, \widetilde{z}_{k+1}, y_k, z_k, \widehat{y}_{k}, \widehat{z}_{k}, \sigma_{k}, \phi_{k}, \psi_{k}$, and $\gamma_{k}$.

\If{$\Phi_{\theta} (x_{k+1}, \widetilde{y}_{k+1}, \widetilde{z}_{k+1}) \leq \frac{1}{2} \phi_k$, $\widetilde{y}_{k+1} \in C$, and $\widetilde{z}_{k+1} \in D$,} \Comment{Step~4.1}
\begin{gather*}
y_{k+1} \coloneqq \widetilde{y}_{k+1}, \quad z_{k+1} \coloneqq \widetilde{z}_{k+1}, \quad \widehat{y}_{k+1} \coloneqq \widetilde{y}_{k+1}, \quad \widehat{z}_{k+1} \coloneqq \widetilde{z}_{k+1}, 
\\
\textstyle \sigma_{k+1} \coloneqq \sigma_{k}, \quad \phi_{k+1} \coloneqq \frac{1}{2}\phi_{k}, \quad \psi_{k+1} \coloneqq \psi_{k}, \quad \gamma_{k+1} \coloneqq \gamma_{k}.
\end{gather*}

\ElsIf{$\Psi_{\theta} (x_{k+1}, \widetilde{y}_{k+1}, \widetilde{z}_{k+1}) \leq \frac{1}{2} \psi_{k}$, $\widetilde{y}_{k+1} \in C$, and $\widetilde{z}_{k+1} \in D$,} \Comment{Step~4.2}
\begin{gather*}
y_{k+1} \coloneqq \widetilde{y}_{k+1}, \quad z_{k+1} \coloneqq \widetilde{z}_{k+1}, \quad \widehat{y}_{k+1} \coloneqq \widetilde{y}_{k+1}, \quad \widehat{z}_{k+1} \coloneqq \widetilde{z}_{k+1}, 
\\
\textstyle \sigma_{k+1} \coloneqq \sigma_{k}, \quad \phi_{k+1} \coloneqq \phi_{k}, \quad \psi_{k+1} \coloneqq \frac{1}{2}\psi_{k}, \quad \gamma_{k+1} \coloneqq \gamma_{k}.
\end{gather*}

\ElsIf{$\Vert \grad F(x_{k+1}; \widehat{y}_{k}, \widehat{z}_{k}, \sigma_{k}) \Vert_{x_{k+1}} \leq \gamma_k$,} \Comment{Step~4.3}
\begin{gather*}
\textstyle y_{k+1} \coloneqq \widehat{y}_{k} - \frac{1}{\sigma_{k}} g(x_{k+1}), \quad z_{k+1} \coloneqq [ \widehat{z}_{k} - \frac{1}{\sigma_{k}} h(x_{k+1}) ]_{+},
\\
\textstyle \widehat{y}_{k+1} \coloneqq P_{C} ( \widehat{y}_{k} - \frac{1}{\sigma_{k}} g(x_{k+1}) ), \quad \widehat{z}_{k+1} \coloneqq P_{D} ( \widehat{z}_{k} - \frac{1}{\sigma_{k}} h(x_{k+1})), 
\\
\textstyle \sigma_{k+1} \coloneqq \min \{ \frac{1}{2} \sigma_{k}, r(x_{k},y_{k},z_{k})^{\frac{3}{2}} \}, \quad \phi_{k+1} \coloneqq \phi_{k}, \quad \psi_{k+1} \coloneqq \psi_{k}, \quad \gamma_{k+1} \coloneqq \frac{1}{2}\gamma_{k}.
\end{gather*}

\Else \Comment{Step~4.4}
\begin{gather*}
y_{k+1} \coloneqq {y}_{k}, \quad z_{k+1} \coloneqq {z}_{k}, \quad \widehat{y}_{k+1} \coloneqq \widehat{y}_{k}, \quad \widehat{z}_{k+1} \coloneqq \widehat{z}_{k}, 
\\
\sigma_{k+1} \coloneqq \sigma_{k}, \quad \phi_{k+1} \coloneqq \phi_{k}, \quad \psi_{k+1} \coloneqq \psi_{k}, \quad \gamma_{k+1} \coloneqq \gamma_{k}.
\end{gather*}
  
\EndIf

\Ensure{$(y_{k+1}, z_{k+1}, \widehat{y}_{k+1}, \widehat{z}_{k+1}, \sigma_{k+1}, {\phi}_{k+1}, \psi_{k+1}, \gamma_{k+1})$.}

\end{algorithmic}
\end{algorithm}

\noindent
From now on, we show the well-definedness of Algorithm~\ref{alg:main}. To this end, the solvability of problem~\eqref{pro:RNLOparch} is verified in the next proposition. The proof is given in Appendix~\ref{sec:appendix}.
\begin{proposition}\label{prop:solution}
Suppose that $M_k$ is self-adjoint and coercive. Then, problem~\eqref{pro:RNLOparch} has a unique optimal solution~$(\xi_{k}^{\ast},\eta_{k}^{\ast}) \in T_{x_{k}} \M \times \R^{\ell}$, and it satisfies the KKT conditions of~\eqref{pro:RNLOparch}. 
\end{proposition}

Lemma~\ref{lem:grad} shows that $\xi_{k}^{\ast}$, which is calculated by Line~\ref{step:subpro} in Algorithm~\ref{alg:main}, is a descent direction of the merit function $F$ if $\grad F(x_k; \widehat{y}_k, \widehat{z}_k, \sigma_k) \not = 0$. The proof is given in Appendix~\ref{sec:appendix}.
\begin{lemma}\label{lem:grad}
Suppose that $M_k$ is self-adjoint and coercive. Then, problem~\eqref{pro:RNLOparch} has a unique optimal solution~$(\xi_{k}^{\ast},\eta_{k}^{\ast})$, and it satisfies
\begin{gather}
\langle \grad F(x_k; \widehat{y}_k, \widehat{z}_k, \sigma_k), \xi_{k}^{\ast} \rangle_{x_k} \leq - \langle M_k[\xi_{k}^{\ast}], \xi_{k}^{\ast} \rangle_{x_k} - \sigma_k \| \eta_{k}^{\ast} - [t_k]_+ \|^{2},\label{eq:gradient}
\\
\grad F(x_k; \widehat{y}_k, \widehat{z}_k, \sigma_k) = - M_k[\xi_{k}^{\ast}] + \rmD h(x_k)^{\ast}[\eta_{k}^{\ast} - [t_{k}]_{+}]. \label{eq:ineqabs}
\end{gather}
Moreover, $\grad F(x_k; \widehat{y}_k, \widehat{z}_k, \sigma_k)=0$ if and only if $(\xi_{k}^{\ast}, \eta_{k}^{\ast}) = (0, [t_k]_+)$ solves problem~\eqref{pro:RNLOparch}.
\end{lemma}

Lemma~\ref{lem:grad} implies that the smallest nonnegative integer $j_k$, determined in the Armijo's line search, can be found within a finite number of iterations. Combining~this fact and Proposition~\ref{prop:solution}, we can verify that Algorithm~\ref{alg:main} is well-defined.
\par
Next, to establish the global convergence property of Algorithm~\ref{alg:main}, we impose the following assumptions.
\begin{assumption}\label{ass:grad}
~
\begin{description}
\item[{\rm (A1)}] The functions $f$, $g$, and $h$ are twice continuously differentiable on $\M$.

\item[{\rm (A2)}] There exists $\nu>0$ such that
\begin{gather}
\textstyle \frac{1}{\nu} \| u \|_{x_k}^2 \leq \langle M_{k}[u], u \rangle_{x_k},
~~ |\langle H_k[u], v \rangle_{x_k}| \leq \nu \| u \|_{x_k} \| v \|_{x_k}
\end{gather}
for all $u \in T_{x_k} \M$, $v \in T_{x_k} \M$, and $k \in \mathbb{N} \cup \{ 0 \}$.
\end{description}
\end{assumption}

\noindent
Under Assumption~\ref{ass:grad}, Proposition~\ref{prop:solution} implies that problem~\eqref{pro:RNLOparch} is solvable, and Lemma~\ref{lem:grad} ensures that Armijo's line-search is well-defined. Therefore, Assumption~\ref{ass:grad} guarantees the well-definedness of Algorithm~\ref{alg:main}. In the following, we suppose that Algorithm~\ref{alg:main} generates an infinite number of iterations.

\par
In the convergence analysis of Algorithm~\ref{alg:main}, we use three mutually disjoint sets ${\cal{I}}, {\cal{J}}$, and ${\cal{K}}$ defined as follows:
\begin{align*}
\mathcal{I} &\coloneqq \{ k \in \mathbb{N} \cup \{ 0 \} ; \, \text{Line~\ref{update:Algo2} of Algorithm~\ref{alg:main} is updated by Step~4.1 or 4.2.} \},
\\
\mathcal{J} &\coloneqq \{ k \in \mathbb{N} \cup \{ 0 \} ; \, \text{Line~\ref{update:Algo2} of Algorithm~\ref{alg:main} is updated by Step~4.3.} \},
\\
\mathcal{K} &\coloneqq \{ k \in \mathbb{N} \cup \{ 0 \} ; \, \text{Line~\ref{update:Algo2} of Algorithm~\ref{alg:main} is updated by Step~4.4.} \}.
\end{align*}
Moreover, we prepare some lemmas for the convergence analysis. 

\begin{lemma}\label{lem:cardIJ}
The algorithm satisfies the following assertions:
\begin{description}
\item[{\rm (i)}] If $\mathrm{card}(\cal{I})=\infty$, then $\phi_k \to 0$ or $\psi_k \to 0$ as $k \to \infty$.
\item[{\rm (ii)}] If $\mathrm{card}(\cal{J})=\infty$, then $\gamma_k \to 0$ and $\sigma_k \to 0$ as $k \to \infty$.
\item[{\rm (iii)}] The sequences $\{ \widehat{y}_k \}$ and $\{ \widehat{z}_k \}$ are included in compact sets $C$ and $D$, respectively.
\end{description}
\end{lemma}

\begin{proof}
First, we prove item~(i). If $\mathrm{card}(\cal{I})=\infty$, the number of iterations updated by either Step 4.1 or 4.2 becomes infinite. If the number of iterations updated by Step 4.1 is infinite, then $\phi_k \to 0$ as $k \to \infty$ because the updating rule implies that $\phi_{k+1} \coloneqq \frac{1}{2}\phi_{k}$ for any $k\in\{k \in \mathbb{N}; \text{Step 4.1 is used}\}$ and $\phi_{k+1} \coloneqq \phi_{k}$ for any $k\in\{k \in \mathbb{N}; \text{Step 4.1 is not used}\}$. If the number of iterations updated by Step 4.2 is infinite, then $\psi_k \to 0$ as $k \to \infty$ because the updating rule ensures that $\psi_{k+1} \coloneqq \frac{1}{2}\psi_{k}$ for any $k\in\{k \in \mathbb{N}; \text{Step 4.2 is used}\}$ and $\psi_{k+1} \coloneqq \psi_{k}$ for any $k\in\{k \in \mathbb{N}; \text{Step 4.2 is not used}\}$. Therefore, item~(i) is satisfied.
\par
Second, we prove item~(ii). For any $k\in{\cal{J}}$, the acceptance test of Step~4.3 is true, that is, $\| \grad F(x_{k+1};\widehat{y}_k,\widehat{z}_k,\sigma_k) \|_{x_{k+1}} \leq \gamma_k$. Therefore, the updating rule implies that $\gamma_{k+1} \coloneqq \frac{1}{2}\gamma_{k} $ and $\sigma_{k+1} \leq \frac{1}{2}\sigma_{k}$ for any $k \in {\cal J}$ and $\gamma_{k+1} \coloneqq \gamma_{k}$ and $\sigma_{k+1} \coloneqq \sigma_{k}$ for any $k \notin {\cal J}$. Thus, if ${\rm card}({\cal J}) = \infty$, then $\gamma_k \to 0$ and $\sigma_k \to 0$ as $k \to \infty$. 
\par
Finally, we prove item~(iii) using mathematical induction. For $k=0$, we can verify that $\widehat{y}_0 = y_0 \in C$ and $\widehat{z}_0 = z_0 \in D$. Next, we assume that $k \in {\mathbb{N}}\cup \{0\}$ satisfies $\widehat{y}_k \in C$ and $\widehat{z}_k \in D$. If $k \in {\cal I}$, it is clear that $\widehat{y}_{k+1} \in C$ and $\widehat{z}_{k+1} \in D$ because $\widehat{y}_{k+1} \coloneqq \widetilde{y}_{k+1} \in C$ and $\widehat{z}_{k+1} \coloneqq \widetilde{z}_{k+1} \in D$ by the updating rule of Steps~4.1 and 4.2. If $k \in {\cal J}$, we readily have $\widehat{y}_{k+1} \in C$ and $\widehat{z}_{k+1} \in D$ because $\widehat{y}_{k+1} \coloneqq P_C(\widehat{y}_k - \frac{1}{\sigma_k} g(x_{k+1})) \in C$ and $\widehat{z}_{k+1} \coloneqq P_D(\widehat{z}_k - \frac{1}{\sigma_k} h(x_{k+1})) \in D$ by the updating rule of Step~4.3. If $k \in {\cal K}$, then $\widehat{y}_{k+1} \in C$ and $\widehat{z}_{k+1} \in D$ hold because $\widehat{y}_{k+1} \coloneqq \widehat{y}_{k} \in C$ and $\widehat{z}_{k+1} \coloneqq \widehat{z}_{k} \in D$ by the updating rule of Step~4.4. As a result, we obtain $\widehat{y}_{k+1} \in C$ and $\widehat{z}_{k+1} \in D$. Thus, item~(iii) is satisfied by mathematical induction.
\end{proof}

\begin{lemma}\label{lem:Step_F}
Suppose that Assumption~{\rm \ref{ass:grad}} is satisfied and that $\{x_k\}$ is bounded. If $\mathrm{card}(\cal{I})<\infty$, $\mathrm{card}(\cal{J})<\infty$, and $\mathrm{card}(\cal{K})=\infty$, then the following statements are satisfied.
\begin{description}
\item[{\rm (i)}] There exist $\widehat{k}\in\mathbb{N}$, $\widehat{y} \in \R^{m}$, $\widehat{z} \in \R^{\ell}$, $\widehat{\sigma} > 0$, and $\hat{\gamma} > 0$ such that $k\in{\cal{K}}$, $\widehat{y}_k = \widehat{y}$, $\widehat{z}_k = \widehat{z}$, $\sigma_k = \widehat{\sigma}$, and $\gamma_k = \widehat{\gamma}$ for all $k \geq \widehat{k}$.

\item[{\rm (ii)}] The sequence $\{ (x_k, p_k) \} \subset T\M$ has a convergent subsequence.

\item[{\rm (iii)}] $\liminf_{k \to \infty} | \langle \grad F(x_k; \widehat{y}_k, \widehat{z}_k, \sigma_k), p_k \rangle_{x_k} | > 0$.
\end{description}
\end{lemma}

\begin{proof}
Since $\mathcal I$ and $\mathcal J$ are finite, whereas $\mathcal K$ is infinite, there exists $\widehat{k} \in \mathbb N$ such that $k \in \mathcal K$ for all $k \geq \widehat{k}$. Consequently, Step~4.4 is performed at every iteration \(k\geq\widehat k\). The corresponding update rules imply that there exist $\widehat y \in \R^m$, $\widehat z\in\R^\ell$, $\widehat\sigma>0$, and $\widehat\gamma>0$ such that $\widehat y_k=\widehat y$, $\widehat z_k=\widehat z$, $\sigma_k=\widehat\sigma$, and $\gamma_k=\widehat\gamma$ for all $k\geq\widehat k$. This proves~{\rm (i)}.
\par
We next prove~{\rm (ii)}. By the boundedness of $\{ x_k \}$, there exist $v \in \M$ and $\eta > 0$ such that $\{ x_k \} \subset B(v, \eta)$. We have from~\cite[Corollary~2.105]{Gallot:2004} that $B(v, \eta)$ is compact. Thus, there exists a convergent subsequence $\{ x_k \}_{k \in {\cal P}} \subset \{ x_k \}$ such that $x_{k} \to x^{\ast}$ as ${\cal P} \ni k \to \infty$. 
Recall that $p_{k} = \xi_{k}^{\ast}$ holds by the updating rule of Algorithm~\ref{alg:main}. From Lemma~\ref{lem:grad}, we can use inequality~\eqref{eq:gradient}. It then follows from (A2) of Assumption~\ref{ass:grad} that
\begin{align*}
\frac{1}{\nu} \| p_k \|_{x_k}^2 
&\leq \langle M_k[p_k], p_k \rangle_{x_k} 
\\
&\leq - \langle \grad F(x_k; \widehat y_{k}, \widehat z_{k}, \sigma_{k}), p_k \rangle_{x_k}
\\
&\leq \sup \{ \|\grad F(x; \widehat y_{k}, \widehat z_{k}, \sigma_{k})\|_{x}; x \in B(v, \eta) \} \|p_k\|_{x_k}
\end{align*}
for all $k \in \mathbb{N}$. Using item~(i) yields $\| p_k \|_{x_k} \leq \nu \sup \{ \|\grad F(x;\widehat y,\widehat z,\widehat\sigma)\|_{x}; x \in B(v, \eta) \}$ for all $k \geq \widehat{k}$. By the compactness of $B(v, \eta)$ and the continuity of $x \mapsto \| \grad F(x;\widehat y,\widehat z,\widehat\sigma) \|_x$, we obtain $\sup \{ \|\grad F(x;\widehat y,\widehat z,\widehat\sigma)\|_{x}; x \in B(v, \eta) \} < \infty$. These facts imply that $\{\|p_k\|_{x_k}\}_{k\geq\widehat k}$ is bounded. Since $\{ \| p_{k} \|_{x_k} \}_{k < \widehat{k}}$ is a finite set, we obtain $\sup \{ \| p_k \|_{x_{k}}; k \in \mathbb{N} \} < \infty$. 
Now, let $({\cal U}, \varphi)$ be a $d$-dimensional chart around $x^{\ast}$. Because $\{ x_k \}_{k \in {\cal P}}$ converges to $x^{\ast}$, there exists $\widetilde{k} \in \mathbb{N}$ such that $x_k \in {\cal U}$ for all $k \in {\cal P}$ with $k \geq \widetilde{k}$. Let $\varphi = (\varphi_1, \ldots, \varphi_d)$. Then, for any $k \in {\cal P}$ with $k \geq \widetilde{k}$, there exists $a_k \in \mathbb{R}^{d}$ such that 
\begin{align}
p_k = \sum_{j=1}^{d} [a_k]_j \frac{\partial}{\partial \varphi_j}\Bigg|_{x_k}. \label{equation:pk}
\end{align}
We define $G \colon {\cal U} \to \mathbb{R}^{d \times d}$ as
\begin{align}
[G(x)]_{ij} \coloneqq \Bigg\langle \frac{\partial}{\partial \varphi_i}\Bigg|_{x}, \frac{\partial}{\partial \varphi_j}\Bigg|_{x} \Bigg\rangle_{x} \quad \forall x \in {\cal U}, ~ \forall (i,j) \in \{1, \ldots, d \} \times \{1, \ldots, d \}. \label{def:Gx}
\end{align}
Using~\eqref{equation:pk} and~\eqref{def:Gx} yields
\begin{align}
\| p_k \|_{x_k}^2 = \sum_{i=1}^{d} \sum_{j=1}^{d} [a_k]_i [a_k]_j \Bigg\langle \frac{\partial}{\partial \varphi_i}\Bigg|_{x_k}, \frac{\partial}{\partial \varphi_j}\Bigg|_{x_k} \Bigg\rangle_{x_k} = a_k^{\top} G(x_k) a_k \label{equation:pk_akGkak}
\end{align}
for all $k \in {\cal P}$ with $k \geq \widetilde{k}$. Note that $G \colon {\cal U} \to \mathbb{R}^{d \times d}$ is continuous at $x^{\ast}$, and $G(x^{\ast})$ is positive definite. Hence, for $\mu \coloneqq \frac{1}{2}\lambda_{\min}(G(x^{\ast})) > 0$, there exists $\delta > 0$ such that $\| G(x) - G(x^{\ast}) \|_2 \leq \mu$ for all $x \in B(x^{\ast}, \delta) \subset {\cal U}$, where $\lambda_{\min}(G(x^{\ast}))$ denotes the smallest eigenvalue of $G(x^{\ast})$. Since $x_k \to x^{\ast}~({\cal P} \ni k \to \infty)$, there exists $\overline{k} \geq \widetilde{k}$ such that $x_k \in B(x^{\ast}, \delta) \subset {\cal U}$ for all $k \in {\cal P}$ with $k \geq \overline{k}$. Then, we have
\begin{align}
a_k^{\top} G(x_k) a_k \geq a_k^{\top} G(x^{\ast}) a_k - \| G(x_k) - G(x^{\ast}) \|_2 \| a_k \|^2 \geq  2 \mu \| a_k \|^2 - \mu \| a_k \|^2 = \mu \| a_k \|^2 \label{equation:akGkak_muak}
\end{align}
for any $k \in {\cal P}$ with $k \geq \overline{k}$. Combining~\eqref{equation:pk_akGkak} and~\eqref{equation:akGkak_muak} yields $\| a_k \| \leq \frac{1}{\sqrt{\mu}} \| p_k \|_{x_k}$ for all $k \in {\cal P}$ with $k \geq \overline{k}$. From this fact and $\sup \{ \| p_k \|_{x_{k}}; k \in \mathbb{N} \} < \infty$, we obtain $\sup \{ \| a_k \|; k \in {\cal P}, k \geq \overline{k} \} < \infty$. From this fact, there exist $a^{\ast} \in \mathbb{R}^{d}$ and ${\cal Q} \subset \{ k \in {\cal P}; k \geq \overline{k} \}$ such that $a_{k} \to a^{\ast}$ as ${\cal Q} \ni k \to \infty$. Now, we denote 
\begin{align}
p^{\ast} \coloneqq \sum_{j=1}^{d} [a^{\ast}]_j \frac{\partial}{\partial \varphi_j} \Bigg|_{x^{\ast}}. \label{equation:past}
\end{align}
Meanwhile, we define $\widetilde{\varphi} \colon T{\cal U} \to \varphi({\cal U}) \times \mathbb{R}^{d}$ as
\begin{align}
\widetilde{\varphi} \left( x, \sum_{j=1}^{d} [a]_j \frac{\partial}{\partial \varphi_j} \Bigg|_{x} \right) = (\varphi(x), a). \label{function:varphi}
\end{align}
Note that $\widetilde{\varphi}$ has the inverse mapping $\widetilde{\varphi}^{-1}$. Note also that $\widetilde{\varphi}$ is a homeomorphism with respect to the standard tangent-bundle topology on $T{\cal U}$. Since $(x_{k}, a_{k}) \to (x^{\ast}, a^{\ast})$ as ${\cal Q} \ni k \to \infty$, it follows from~\eqref{equation:pk},~\eqref{equation:past}, and~\eqref{function:varphi} that $\widetilde{\varphi}(x_{k}, p_{k}) = (\varphi(x_{k}), a_{k}) \to (\varphi(x^{\ast}), a^{\ast}) = \widetilde{\varphi}(x^{\ast}, p^{\ast})$ as ${\cal Q} \ni k \to \infty$. This fact and the continuity of $\widetilde{\varphi}^{-1}$ imply $(x_{k}, p_{k}) \to (x^{\ast}, p^{\ast})$ as ${\cal Q} \ni k \to \infty$. This shows item~(ii) because the topology on $T\mathcal{U}$ is the subspace topology inherited from $T\mathcal{M}$.
\par
It remains to prove item~(iii). Suppose, to the contrary, that
\begin{align*}
\liminf_{k \to \infty}
\left|
\langle \grad F(x_k; \widehat y_k, \widehat z_k, \sigma_k), p_k \rangle_{x_k}
\right|=0.
\end{align*}
By item~(i), there exists an infinite index set ${\cal R} \subset \{ k \in \mathbb{N}; k \geq \widehat k\}$ such that
\begin{align}
\lim_{\mathcal R \ni k\to\infty}
\left| \langle\grad F(x_k;\widehat y,\widehat z,\widehat\sigma),p_k\rangle_{x_k} \right|=0.
\label{eq:lemStepF-vanishing-product}
\end{align}
By the compactness argument used in the proof of item~(ii), the sequential compactness of $\{ x_k \}_{k \in {\cal R}}$ has been ensured, that is, there exist $x^{\ast} \in \M$ and ${\cal S} \subset {\cal R}$ such that $x_k \to x^{\ast}$ as ${\cal S} \ni k \to \infty$. Let $k \in {\cal S}$ be arbitrary. 
Using equality~\eqref{eq:ineqabs}, (A2) of Assumption~\ref{ass:grad}, and the definition of $M_k$ gives
\begin{align}
\begin{aligned} \label{eq:lemStepF-grad-zero}
\|\grad F(x_k;\widehat y,\widehat z,\widehat\sigma)\|_{x_k} 
&\leq \left( \nu + \frac{1}{\widehat{\sigma}} \| \rmD g(x_k) \|^2 \right) \| p_k \|_{x_k} + \| \rmD h(x_k) \| \| \eta_k^{\ast} - [t_k]_{+} \|. 
\end{aligned}
\end{align}
Combining~\eqref{eq:gradient} and~\eqref{eq:lemStepF-vanishing-product} yields
\begin{align}
\lim_{{\cal S} \ni k \to \infty} \| p_k \|_{x_k} = 0, \quad \lim_{{\cal S} \ni k \to \infty} \| \eta_k^{\ast} - [t_k]_+ \| = 0. \label{eq:lemStepF-p-eta-zero}
\end{align}
Because $\{ x_k \}_{k \in {\cal S}} \subset B(v, \eta)$ is satisfied, the continuity of $\rmD g$ and $\rmD h$ implies that $\{ \| \rmD g(x_k) \| \}_{k \in {\cal S}}$ and $\{ \| \rmD h(x_k) \| \}_{k \in {\cal S}}$ are bounded. It then follows from~\eqref{eq:lemStepF-grad-zero},~\eqref{eq:lemStepF-p-eta-zero}, and the continuity of $\M \ni x \mapsto \| \grad F(x, \widehat{y}, \widehat{z}, \widehat{\sigma}) \|_{x} \in \R$ that
\begin{align}
\| \grad F(x^{\ast};\widehat y,\widehat z,\widehat\sigma) \|_{x^{\ast}} = \lim_{{\cal S} \ni k \to \infty} \| \grad F(x_k;\widehat y,\widehat z,\widehat\sigma)\|_{x_k} = 0. \label{lim:gradF1}
\end{align}
Now, we recall that $\{ \alpha_k \} \subset (0,1]$ and $x_{k} \to x^{\ast}$ as ${\cal S} \ni k \to \infty$. Hence, the continuity of the retraction and~\eqref{eq:lemStepF-p-eta-zero} imply $x_{k+1} = R_{x_k}(\alpha_k p_k)\to x^{\ast}$ as ${\cal S} \ni k \to \infty$. Therefore, we have
\begin{align}
\lim_{{\cal S} \ni k \to \infty} \|\grad F(x_{k+1};\widehat y,\widehat z,\widehat\sigma)\|_{x_{k+1}} = \| \grad F(x^{\ast};\widehat y,\widehat z,\widehat\sigma) \|_{x^{\ast}}. \label{lim:gradF2}
\end{align}
We have from~\eqref{lim:gradF1} and~\eqref{lim:gradF2} that $\|\grad F(x_{k+1};\widehat y,\widehat z,\widehat\sigma)\|_{x_{k+1}} \to 0$ as ${\cal S} \ni k \to \infty$. Thus, there exists $\widetilde{k} \in {\cal S} \subset {\cal R} \subset \{ k \in \mathbb{N}; k \geq \widehat{k} \}$ such that $\|\grad F(x_{k+1};\widehat y,\widehat z,\widehat\sigma)\|_{x_{k+1}} \leq \widehat{\gamma}$ for all $k \geq \widetilde{k}$ with $k \in {\cal S}$. In other words, the acceptance condition in Step~4.3 is satisfied for all $k \geq \widetilde{k}$ with $k \in {\cal S}$. This contradicts the fact that Step~4.4 is performed for all $k \geq \widehat k$. Therefore, item~(iii) is proven.
\end{proof}

In what follows, we show that the algorithm cannot remain in Step~4.4 at every sufficiently large iteration.

\begin{proposition}\label{lem:notinf}
Suppose that Assumption~{\rm \ref{ass:grad}} is satisfied and that $\{x_k\}$ is bounded. Then, there does not occur a situation such that $\mathrm{card}(\cal{I})<\infty$, $\mathrm{card}(\cal{J})<\infty$, and $\mathrm{card}(\cal{K})=\infty$.
\end{proposition}

\begin{proof}
This assertion is shown by contradiction. Suppose that $\mathrm{card}(\cal{I})<\infty$, $\mathrm{card}(\cal{J})<\infty$, and $\mathrm{card}(\cal{K})=\infty$. From item~(i) of Lemma~\ref{lem:Step_F}, there exist $\hat{k}\in\mathbb{N}$, $\widehat{y} \in \R^{m}$, $\widehat{z} \in \R^{\ell}$, $\widehat{\sigma} > 0$, and $\widehat{\gamma} > 0$ such that $k \in \cal{K}$, $\widehat{y}_k = \widehat{y}$, $\widehat{z}_k = \widehat{z}$, $\sigma_k = \widehat{\sigma}$, and $\gamma_k = \widehat{\gamma}$ for all $k \geq \widehat{k}$. In the following, we assume that $k \geq \widehat{k}$. From the boundedness of $\{ x_k \}$, there exist $v \in \M$ and $\eta >0$ such that $\{ x_k \} \subset B(v, \eta)$. Since $\M$ is connected and complete, we have from~\cite[Corollary~2.105]{Gallot:2004} that $B(v, \eta)$ is compact. Then, it is clear that 
\begin{align}
| F(x_{k}; \widehat{y}, \widehat{z}, \widehat{\sigma}) | \leq \sup \{ |F(x; \widehat{y}, \widehat{z}, \widehat{\sigma})| ; x \in B(v,\eta) \} < \infty. \label{bounded:F_value}
\end{align}
Using $x_{k+1} = R_{x_k}(\beta^{j_k} p_k)$,~\eqref{eq:armijo-rule}, and~\eqref{eq:gradient} yields
\begin{gather}
\langle \grad F(x_k; \widehat{y}, \widehat{z}, \widehat{\sigma}), p_k \rangle_{x_k} \leq 0, \label{ineq:descent_direction}
\\
-\varepsilon \beta^{j_k} \langle \grad F(x_k; \widehat{y}, \widehat{z}, \widehat{\sigma}), p_k \rangle_{x_k} \leq F(x_k; \widehat{y}, \widehat{z}, \widehat{\sigma}) - F(x_{k+1}; \widehat{y}, \widehat{z}, \widehat{\sigma}). \label{ineq:Armijo_ineq}
\end{gather}
Thus, these inequalities imply $F(x_{k+1}; \widehat{y}, \widehat{z}, \widehat{\sigma}) \leq F(x_{k}; \widehat{y}, \widehat{z}, \widehat{\sigma})$. This fact and~\eqref{bounded:F_value} yield that $\{ F(x_k; \widehat{y}, \widehat{z}, \widehat{\sigma}) \}$ is bounded and monotonically nonincreasing, that is, it is convergent. It then follows from~\eqref{ineq:Armijo_ineq} that
\begin{align}
\lim_{k \to \infty} \beta^{j_k} \langle \grad F(x_k; \widehat{y}, \widehat{z}, \widehat{\sigma}), p_k \rangle_{x_k} = 0.
\end{align}
From this equality and item~(iii) of Lemma~\ref{lem:Step_F}, we have $\lim_{k \to \infty} \beta^{j_k} = 0$. By combining this fact and item~(ii) of Lemma~\ref{lem:Step_F}, there exist $x^{\ast} \in \M$, $p^{\ast} \in T_{x^{\ast}}{\cal M}$, and ${\cal N} \subset \mathbb{N}$ such that $j_k \to \infty$ and $(x_{k}, p_{k}) \to (x^{\ast}, p^{\ast})$ in $T{\cal M}$ as ${\cal N} \ni k \to \infty$. Now, let $k \in {\cal N}$, and assume without loss of generality that $j_{k} \geq 1$. For simplicity, we denote 
\begin{align*}
\delta_{k} \coloneqq \beta^{j_{k}-1}, \quad
\widehat{F}(\delta; x_k, p_k) \coloneqq F(R_{x_k}(\delta p_k); \widehat{y}, \widehat{z}, \widehat{\sigma}) \quad \forall \delta \geq 0.
\end{align*}
The chain rule of the derivative leads to
\begin{align}
\widehat{F}^{\prime}(\delta; x_k, p_k) = \langle \grad F(R_{x_k}(\delta p_k); \widehat{y}, \widehat{z}, \widehat{\sigma}), \rmD R_{x_k}(\delta p_k)[p_k] \rangle_{R_{x_k}(\delta p_k)} \quad \forall \delta \geq 0. \label{eq:Fprime_delta}
\end{align}
Note that $R_{x_k}(0) = x_k$ and that $\rmD R_{x_k}(0)$ is the identity mapping on $T_{x_k} \M$. From these facts, we also have
\begin{align}
\widehat{F}^{\prime}(0; x_k, p_k) = \langle \grad F(x_k; \widehat{y}, \widehat{z}, \widehat{\sigma}), p_k \rangle_{x_k}, \label{eq:Fprime0}
\end{align}
Then, Armijo's condition fails for $\delta_{k}$, that is,
\begin{align*}
\widehat{F}(0; x_k, p_k) + \varepsilon \delta_k \widehat{F}^{\prime}(0; x_k, p_k) = \widehat{F}(0; x_k, p_k) + \varepsilon \delta_k \langle \grad F(x_k; \widehat{y}, \widehat{z}, \widehat{\sigma}), p_k \rangle_{x_k} < \widehat{F}(\delta_k; x_k, p_k).
\end{align*}
It then follows from $\varepsilon \in (0,1)$,~\eqref{ineq:descent_direction}, and~\eqref{eq:Fprime0} that
\begin{align}
0 \leq (\varepsilon - 1) \widehat{F}^{\prime}(0; x_k, p_k) \leq \frac{\widehat{F}(\delta_k; x_k, p_k) - \widehat{F}(0; x_k, p_k)}{\delta_k} - \widehat{F}^{\prime}(0; x_k, p_k). \label{ineq:Fprime}
\end{align}
The mean value theorem guarantees $(\widehat{F}(\delta_k; x_k, p_k)-\widehat{F}(0; x_k, p_k))/\delta_k  = \widehat{F}^{\prime}(\theta_k\delta_k; x_k, p_k)$ for some $\theta_k\in(0,1)$. Combining this fact,~\eqref{eq:Fprime0}, and~\eqref{ineq:Fprime} yields
\begin{align}
0 \leq (\varepsilon - 1) \langle \grad F(x_k; \widehat{y}, \widehat{z}, \widehat{\sigma}), p_k \rangle_{x_k} \leq \widehat{F}^{\prime}(\theta_k \delta_k; x_k, p_k) - \widehat{F}^{\prime}(0; x_k, p_k). \label{ineq:gradF_xk}
\end{align}
Now, exploiting~\eqref{eq:Fprime_delta} leads to
\begin{align}
\begin{aligned} \label{eq:Fprime_thetadelta}
& \widehat{F}^{\prime}(\theta_k \delta_k; x_k, p_k) - \widehat{F}^{\prime}(0; x_k, p_k)
\\
& = \langle \grad F(R_{x_k}(\theta_k \delta_k p_k); \widehat{y}, \widehat{z}, \widehat{\sigma}), \rmD R_{x_k}(\theta_k \delta_k p_k)[p_k] \rangle_{R_{x_k}(\theta_k \delta_k p_k)} 
\\
& \hspace{65mm} - \langle \grad F(x_k; \widehat{y}, \widehat{z}, \widehat{\sigma}), p_k \rangle_{x_k}.
\end{aligned}
\end{align}
Recall that $(x_{k}, \theta_{k}\delta_{k}p_{k}) \to (x^{\ast}, 0)$ in $T{\cal M}$ as ${\cal N} \ni k \to \infty$.
Then, taking ${\cal N} \ni k \to \infty$ in~\eqref{eq:Fprime_thetadelta} implies
\begin{align}
\lim_{{\cal N} \ni k \to \infty} \left( \widehat{F}^{\prime}(\theta_k \delta_k; x_k, p_k) - \widehat{F}^{\prime}(0; x_k, p_k) \right) = 0. \label{lim:Fprime_theta}
\end{align}
By using~\eqref{ineq:gradF_xk} and~\eqref{lim:Fprime_theta}, we obtain
\begin{align*}
\lim_{{\cal N} \ni k \to \infty} \langle \grad F(x_k; \widehat{y}, \widehat{z}, \widehat{\sigma}), p_k \rangle_{x_k} = 0.
\end{align*}
This fact contradicts item (iii) of Lemma~\ref{lem:Step_F}. Therefore, the assertion is proven.
\end{proof}

To establish the main convergence theorem, we first discuss convergence separately for the cases ${\rm card}({\cal I}) = \infty$ and ${\rm card}({\cal I}) < \infty$.

\begin{proposition}\label{prop:KKT}
Suppose that Assumption~{\rm \ref{ass:grad}} is satisfied. Let $\{x_k\}$ be a bounded sequence. If $\mathrm{card}(\mathcal{I})=\infty$, then $\{ x_k \}$ has an accumulation point $x^{\ast} \in \M$, and $x^{\ast}$ satisfies the KKT conditions of~\eqref{pro:RNLO}.
\end{proposition}

\begin{proof}
Let $\mathcal{I}_{\Phi}$ and $\mathcal{I}_{\Psi}$ denote the sets of indices at which Steps~4.1 and~4.2, respectively, are performed. Thus, we can easily see that $\mathcal{I}=\mathcal{I}_{\Phi}\cup\mathcal{I}_{\Psi}$ and $\mathcal{I}_{\Phi}\cap\mathcal{I}_{\Psi}=\emptyset$.
Since $\mathrm{card}(\mathcal{I})=\infty$, at least one of $\mathcal{I}_{\Phi}$ and $\mathcal{I}_{\Psi}$ is infinite. Now, suppose first that $\mathrm{card}(\mathcal{I}_{\Phi})=\infty$, and define $\mathcal{P}_{\Phi}\coloneqq\{k \in \mathbb{N};\ k-1 \in \mathcal{I}_{\Phi}\}$. 
For every $k \in {\cal P}_{\Phi}$, Step~4.1 gives
\begin{gather}
y_k = \widetilde y_k \in C \subset \R^{m}, \quad z_k =\widetilde z_k \in D \subset \R^{\ell}_{+}, \quad \Phi_{\theta}(x_k, y_k, z_k)  \leq \frac{1}{2}\phi_{k-1} = \phi_k, \label{eq:propKKT-residual-Phi}
\end{gather}
where we recall that the sets $C$ and $D$ are compact. Meanwhile, from the boundedness of $\{ x_k \}$, there exist $v \in \M$ and $\eta > 0$ such that $\{ x_k \} \subset B(v, \eta)$. We recall that $\M$ is complete and connected. It then follows from~\cite[Corollary~2.105]{Gallot:2004} that $B(v, \eta)$ is compact. Hence, the sequence $\{ x_k \}_{k \in {\cal P}_{\Phi}} \subset \M$ is sequentially compact. By combining this fact and~\eqref{eq:propKKT-residual-Phi}, there exist ${\cal Q} \subset {\cal P}_{\Phi}$, $x^{\ast} \in \M$, $y^{\ast}\in C \subset \R^{m}$, and $z^{\ast}\in D \subset \R^{\ell}_{+}$ such that $x_k \to x^{\ast}$, $y_k \to y^{\ast}$, and $z_k \to z^{\ast}$ as ${\cal Q} \ni k \to \infty$.
Noting ${\rm card}({\cal I}_{\Phi}) = \infty$ and the updating rule for $\phi_k$ yields that $\phi_k \to 0$ as $k \to \infty$. Hence, using~\eqref{eq:propKKT-residual-Phi} and the continuity of $\Phi_{\theta}$ implies $\Phi_{\theta}(x^{\ast},y^{\ast},z^{\ast})=0$. From the definition of $\Phi_{\theta}$ and $\theta\in(0,1)$, the KKT conditions of~\eqref{pro:RNLO} hold at $x^{\ast}$.
\par
On the other hand, the case of $\mathrm{card}(\mathcal{I}_{\Psi})=\infty$ can be shown in a way similar to the case of $\mathrm{card}(\mathcal{I}_{\Phi})=\infty$. 
As a result, the assertion holds.
\end{proof}

For later use, we define the constraint-violation function $J \colon \M \to \R$ by
\begin{align*}
J(x) \coloneqq \frac{1}{2}\|g(x)\|^{2} + \frac{1}{2}\|[-h(x)]_{+}\|^{2}.
\end{align*}
Since the scalar function \(t\mapsto \frac12[-t]_{+}^{2}\) is continuously differentiable, the function $J$ is continuously differentiable on $\M$, and
\begin{align}
\grad J(x)
 = \rmD g(x)^{\ast}[g(x)]
   - \rmD h(x)^{\ast}\big[ [-h(x)]_{+} \big].
\label{eq:grad-feasibility-measure}
\end{align}
Thus, a point \(x\in\M\) is stationary for the feasibility problem
\begin{align}
\begin{aligned} \label{pro:feasibility}
& \mini_{x\in\M} && J(x)
\end{aligned}
\end{align}
if and only if $\grad J(x)=0$.

\begin{proposition}\label{prop:AKKT}
Suppose that Assumption~{\rm \ref{ass:grad}} is satisfied. Let $\{ x_k \}$ be a bounded sequence. If ${\rm card}({\cal I}) < \infty$ holds, then $\{ x_k \}$ has an accumulation point $x^{\ast} \in \M$, and $x^{\ast}$ satisfies either of the following assertions:
\begin{description}
\item[{\rm (i)}] $x^{\ast}$ is an AKKT point of~\eqref{pro:RNLO}, and $\{ (x_k, y_k, z_k) \}$ has a subsequence that is an AKKT sequence corresponding to $x^{\ast}$;

\item[{\rm (ii)}] $x^{\ast}$ is an infeasible point of~\eqref{pro:RNLO}, but is a
stationary point of~\eqref{pro:feasibility}, that is, $\grad J(x^{\ast})=0$.
\end{description}
\end{proposition}

\begin{proof}
To begin with, we show that ${\rm card}({\cal J}) = \infty$. Assume to the contrary that ${\rm card}({\cal J}) < \infty$. Since $\mathcal{I}$, $\mathcal{J}$, and $\mathcal{K}$ form a partition of $\mathbb{N}\cup\{0\}$, the assumptions $\mathrm{card}(\mathcal{I})<\infty$ and $\mathrm{card}(\mathcal{J})<\infty$ would imply $\mathrm{card}(\mathcal{K})=\infty$. But, this case is excluded by Proposition~\ref{lem:notinf}. Consequently, it is clear that $
\mathrm{card}(\mathcal{J})=\infty$. 
\par
Now, we define ${\cal P} \coloneqq \{ k \in \mathbb{N};\ k-1 \in \mathcal{J} \}$. Note that $\{ x_k \}_{k \in {\cal P}} \subset \M$ is sequentially compact. Thus, there exist $x^{\ast} \in \M$ and ${\cal Q} \subset {\cal P}$ such that $x_k \to x^{\ast}$ as ${\cal Q} \ni k \to \infty$. Let $k \in \mathcal{Q}$ be arbitrary. Noting $k-1 \in {\cal J}$ gives
\begin{align}
y_k &=\widehat y_{k-1}-\frac{1}{\sigma_{k-1}}g(x_k), \label{eq:propAKKT-yk}
\\
z_k &=\left[\widehat z_{k-1}-\frac{1}{\sigma_{k-1}}h(x_k)\right]_{+}.
\label{eq:propAKKT-zk}
\end{align}
From~\eqref{eq:graddef},~\eqref{eq:propAKKT-yk}, and~\eqref{eq:propAKKT-zk}, we have $\grad_x L(x_k,y_k,z_k) = \grad F(x_k;\widehat y_{k-1},\widehat z_{k-1},\sigma_{k-1})$. This equality and the acceptance test in Step~4.3 imply $\| \grad_x L(x_k,y_k,z_k) \|_{x_k} \leq\gamma_{k-1}$. Since $\gamma_{k-1} \to 0~(k \to \infty)$ from item~(ii) of Lemma~\ref{lem:cardIJ}, we have
\begin{align}
\lim_{\mathcal{Q} \ni k \to \infty} \| \grad_x L(x_k,y_k,z_k) \|_{x_k} = 0. \label{eq:propAKKT-stationarity}
\end{align}
It remains to verify the asymptotic complementarity condition. By item~(iii) of Lemma~\ref{lem:cardIJ}, there exists $c > 0$ such that
\begin{align}
\big| [\widehat z_k]_j \big| \leq c \quad \forall k \in \mathbb{N} \cup \{ 0 \}, ~ \forall j \in \{ 1,\ldots,\ell \}.
\label{eq:propAKKT-hatz-bound}
\end{align}
Let $j\in\{1, \ldots, \ell\}$. Now, we consider the two cases: (a) $x^{\ast}$ is feasible for~\eqref{pro:RNLO}; (b) otherwise.
\par
Case (a): If $h_j(x^{\ast})>0$, then there exists $\widetilde{n} \in \mathbb{N}$ such that $h_j(x_k) > 0$ for all $k \geq \widetilde{n}$ with $k \in {\cal Q}$. Since $\sigma_{k-1} \to 0~(k \to \infty)$ from item~(ii) of Lemma~\ref{lem:cardIJ}, noting~\eqref{eq:propAKKT-hatz-bound} gives $[\widehat z_{k-1}]_j - \frac{1}{\sigma_{k-1}} h_j(x_k) < 0$ for all $k \geq \widetilde{n}$ with $k \in {\cal Q}$. Hence, using~\eqref{eq:propAKKT-zk} implies $[z_k]_j=0$ for all $k \geq \widetilde{n}$ with $k \in \mathcal{Q}$, and therefore $[z_k]_j[h_j(x_k)]_{+} \to 0$ as $\mathcal{Q} \ni k \to \infty$.
\par
If $h_j(x^{\ast}) = 0$, then~\eqref{eq:propAKKT-zk} and~\eqref{eq:propAKKT-hatz-bound} give $0 \leq [z_k]_j [h_j(x_k)]_{+} \leq c [h_j(x_k)]_{+}$ for all $k \in \mathcal{Q}$. Since $h_j(x_k) \to 0$, we obtain $[z_k]_j[h_j(x_k)]_{+} \to 0$ as $\mathcal{Q} \ni k\to\infty$. From these facts, we conclude that
\begin{align}
\lim_{\mathcal{Q} \ni k \to \infty} \langle z_k, [h(x_k)]_{+} \rangle = 0. \label{eq:propAKKT-complementarity}
\end{align}
Finally, it is clear that $x_k\to x^{\ast}~({\cal Q} \ni k \to \infty)$ and $z_k\in\R^{\ell}_{+}$ for $k \in \mathcal{Q}$. Together with~\eqref{eq:propAKKT-stationarity} and~\eqref{eq:propAKKT-complementarity}, this proves that $\{(x_k,y_k,z_k)\}_{k \in \mathcal{Q}}$ is an AKKT sequence corresponding to $x^{\ast}$. This result shows that item~(i) is satisfied.
\par
Case (b): The positive homogeneity of the projection onto $\R^{\ell}_{+}$ gives $\sigma_{k-1} [\widehat z_{k-1}-\frac{1}{\sigma_{k-1}}h(x_k)]_{+}
 = [\sigma_{k-1}\widehat z_{k-1}-h(x_k)]_{+}$. Hence, multiplying~\eqref{eq:graddef} by $\sigma_{k-1}$ and exploiting~\eqref{eq:grad-feasibility-measure}, we obtain
\begin{align}
\begin{aligned} \label{eq:feas-grad-identity}
& \grad J(x_k) -\sigma_{k-1}\grad F(x_k;\widehat{y}_{k-1},\widehat{z}_{k-1},\sigma_{k-1})
\\
& = -\sigma_{k-1}\grad f(x_k) +\sigma_{k-1}\rmD g(x_k)^{\ast}[\widehat{y}_{k-1}]
\\
&\qquad\quad +\rmD h(x_k)^{\ast} \big[ [\sigma_{k-1}\widehat{z}_{k-1}-h(x_k)]_{+} - [-h(x_k)]_{+} \big].
\end{aligned} 
\end{align}
The nonexpansiveness of the projection $\mathbb{R}^{\ell} \ni v \mapsto [v]_+ \in \mathbb{R}_+^{\ell}$ leads to
\begin{align}
\|[\sigma_{k-1}\widehat z_{k-1}-h(x_k)]_{+}-[-h(x_k)]_{+}\| \leq \sigma_{k-1}\|\widehat z_{k-1}\|. \label{eq:projection-nonexpansive}
\end{align}
Since the acceptance test in Step~4.3 is true at the $(k-1)$-th iteration from $k \in {\cal Q} \subset {\cal P}$, it is clear that $\|\grad F(x_k;\widehat y_{k-1},\widehat z_{k-1},\sigma_{k-1})\|_{x_k} \leq \gamma_{k-1}$. It then follows from~\eqref{eq:feas-grad-identity} and \eqref{eq:projection-nonexpansive} that
\begin{align}
\begin{aligned} \label{eq:feas-grad-bound}
\|\grad J(x_k)\|_{x_k}
&\leq
\sigma_{k-1} \gamma_{k-1} + \sigma_{k-1} \| \grad f(x_{k}) \|_{x_k}
\\
& \qquad \qquad \qquad + \sigma_{k-1} \| \rmD g(x_{k}) \| \| \widehat{y}_{k-1} \| + \sigma_{k-1} \| \rmD h(x_k) \| \| \widehat{z}_{k-1} \|,
\end{aligned}
\end{align}
where $\| \rmD g(x_k) \|$ and $\| \rmD h(x_k) \|$ denote the operator norms of $\rmD g(x_k)$ and $\rmD h(x_{k})$, respectively. By items~(ii) and~(iii) of Lemma~\ref{lem:cardIJ}, we recall that $\sigma_{k-1} \to 0~(k \to \infty)$, and $\{ \| \widehat y_k \| \}$ and $\{ \| \widehat z_k \| \}$ are bounded. In addition, because $x_k\to x^{\ast}~({\cal Q} \ni k \to \infty)$ and (A1) of Assumption~\ref{ass:grad} holds, the boundedness of $\{ \|\grad f(x_k)\|_{x_k} \}_{k \in {\cal Q}}$, $\{ \| \rmD g(x_k) \| \}_{k \in {\cal Q}}$, and $\{ \| \rmD h(x_k) \| \}_{k \in {\cal Q}}$ is ensured. Therefore, combining these facts and~\eqref{eq:feas-grad-bound} yields $\|\grad J(x_k)\|_{x_k} \to 0$ as ${\cal Q} \ni k\to\infty$. Since $\{ x_k \}_{k \in {\cal Q}}$ converges to $x^{\ast}$, the continuity of $\M \ni x \mapsto\|\grad J(x)\|_x \in \R$ gives $\grad J(x^{\ast})=0$. This fact implies that item~(ii) holds.
\end{proof}

Finally, we provide the main convergence theorem by exploiting the above arguments.

\begin{theorem}\label{thm:global}
Suppose that Assumption~{\rm \ref{ass:grad}} is satisfied. Suppose also that the sequence $\{ x_k \}$ generated by Algorithm~{\rm \ref{alg:main}} is bounded. Then, the sequence $\{ x_k \}$ has at least one accumulation point $x^{\ast} \in \M$, which satisfies one of the following assertions:
\begin{description}
\item[{\rm (i)}] $x^{\ast}$ is a KKT point of~\eqref{pro:RNLO};

\item[{\rm (ii)}] $x^{\ast}$ is an AKKT point of~\eqref{pro:RNLO}, and $\{ (x_k, y_k, z_k) \}$ has a subsequence that is an AKKT sequence corresponding to $x^{\ast}$;

\item[{\rm (iii)}] $x^{\ast}$ is an infeasible point of~\eqref{pro:RNLO}, but is a
stationary point of~\eqref{pro:feasibility}, that is, $\grad J(x^{\ast})=0$.
\end{description}
Moreover, if $x^{\ast}$ satisfies the ERCQ, then it is a KKT point of~\eqref{pro:RNLO}.
\end{theorem}

\begin{proof}
We notice that the following two cases can be considered: (a) ${\rm card}({\cal I}) = \infty$; (b) ${\rm card}({\cal I}) < \infty$. Hence, the first part of the theorem follows directly from Propositions~\ref{prop:KKT} and~\ref{prop:AKKT}.
\par
It remains to prove the last part of the theorem. From the above arguments, we only need to consider case (b). Hence, assume that ${\rm card}({\cal I}) < \infty$ in the following. By using Proposition~\ref{prop:AKKT}, there exist $x^{\ast} \in \M$ and ${\cal P} \subset \mathbb{N}$ such that $x_k \to x^{\ast}$ as ${\cal P} \ni k \to \infty$, and either assertion~(ii) or~(iii) holds. Moreover, recall that $x^{\ast}$ satisfies the ERCQ. From now on, we will show that assertion~(ii) must hold, that is, $x^{\ast}$ is an AKKT point of~\eqref{pro:RNLO}, and $\{ (x_k, y_k, z_k) \}$ has a subsequence that is an AKKT sequence corresponding to $x^{\ast}$. To this end, we assume, to the contrary, that assertion~(iii) holds, that is, $x^{\ast}$ is infeasible for~\eqref{pro:RNLO} and satisfies $\grad J(x^{\ast}) = 0$. Since $x^{\ast}$ satisfies the ERCQ, there exist $\widehat{\xi} \in T_{x^{\ast}}\M$ and $\widehat w\in\R^{\ell}_{+}$ such that
\begin{align}
-g(x^{\ast}) &= \rmD g(x^{\ast})[\widehat{\xi}],
\label{eq:prop-feas-ercq-g}
\\
\widehat w &= h(x^{\ast}) + \rmD h(x^{\ast})[\widehat{\xi}].
\label{eq:prop-feas-ercq-h}
\end{align}
Combining $[-h(x^{\ast})]_{+} \geq 0$, $\widehat w \geq 0$, and~\eqref{eq:prop-feas-ercq-h} yields
\begin{align}
0
&\geq-\langle [-h(x^{\ast})]_{+},\widehat w\rangle =-\langle [-h(x^{\ast})]_{+},\rmD h(x^{\ast})[\widehat\xi]\rangle + \| [-h(x^{\ast})]_{+} \|^{2},
\label{eq:prop-feas-ercq-ineq}
\end{align}
where we used $\langle [-h(x^{\ast})]_{+}, h(x^{\ast}) \rangle = -\| [-h(x^{\ast})]_{+} \|^{2}$. Meanwhile, using $\grad J(x^{\ast}) = 0$,~\eqref{eq:grad-feasibility-measure}, and~\eqref{eq:prop-feas-ercq-g} implies $0 = \langle\grad J(x^{\ast}),\widehat\xi\rangle_{x^{\ast}} = \langle g(x^{\ast}),\rmD g(x^{\ast})[\widehat\xi]\rangle -\langle [-h(x^{\ast})]_{+},\rmD h(x^{\ast})[\widehat\xi] \rangle = -\|g(x^{\ast})\|^{2} -\langle [-h(x^{\ast})]_{+}, \rmD h(x^{\ast})[\widehat\xi] \rangle$,
that is, $-\langle [-h(x^{\ast})]_{+}, \rmD h(x^{\ast})[\widehat\xi] \rangle = \|g(x^{\ast})\|^{2}$. By substituting this equality into~\eqref{eq:prop-feas-ercq-ineq}, we obtain $0 \geq \|g(x^{\ast})\|^{2}+\|[-h(x^{\ast})]_{+}\|^{2}$, namely, $x^{\ast}$ is feasible for~\eqref{pro:RNLO}. However, this contradicts assertion~(iii).
\par
From the above argument, we can verify that assertion~(ii) holds. Hence, let $\{ (x_k, y_k, z_k) \}_{k \in {\cal Q}}$ be an AKKT sequence corresponding to the AKKT point $x^{\ast}$. Then, it is clear that
\begin{align}
\lim_{{\cal Q} \ni k \to \infty} \| \grad_x L(x_k, y_k, z_k) \|_{x_k} = 0, \quad \lim_{{\cal Q} \ni k \to \infty} \langle z_k, [h(x_k)]_{+} \rangle = 0. \label{equation:gradL_zh}
\end{align}
In what follows, we will show that $\{(y_k, z_k)\}_{k \in {\cal Q}}$ is bounded. Suppose, to the contrary, that it is unbounded. Let us define $\rho_k \coloneqq \sqrt{\|y_k\|^{2}+\|z_k\|^{2}}$ for each $k \in \mathbb{N} \cup \{ 0 \}$. Then, there exists ${\cal R} \subset {\cal Q}$ such that $\{ \rho_k \}_{k \in {\cal R}} \subset \mathbb{R}$ satisfies $\rho_k \to \infty$ as ${\cal R} \ni k \to \infty$. From the definition of $\{ \rho_k \}$, there exist $\overline{y} \in \R^{m}$, $\overline{z} \in \R^{\ell}_{+}$, and ${\cal S} \subset {\cal R}$ such that
\begin{align}
\lim_{{\cal S} \ni k \to \infty} \frac{y_k}{\rho_k} = \overline{y},
\quad
\lim_{{\cal S} \ni k \to \infty} \frac{z_k}{\rho_k} = \overline{z},
\quad
\| \overline{y} \|^{2} + \| \overline{z} \|^{2} = 1.
\label{eq:thm-normalized-multipliers}
\end{align}
Let $k \in {\cal S}$ be arbitrary. Without loss of generality, we can assume that $\rho_k > 0$. It follows from~\eqref{equation:gradL_zh} and $\rho_k \to \infty~({\cal S} \ni k \to \infty)$ that
\begin{align*}
\lim_{{\cal S} \ni k \to \infty} \Bigg\| \frac{\grad f(x_k)}{\rho_k} - \rmD g(x_k)^{\ast} \left[ \frac{y_k}{\rho_k} \right] - \rmD h(x_k)^{\ast}\left[ \frac{z_k}{\rho_k} \right] \Bigg\|_{x_k} = 0,
\quad \lim_{{\cal S} \ni k \to \infty} \left\langle \frac{z_k}{\rho_k}, [h(x_k)]_+ \right\rangle = 0.
\end{align*} 
Using~\eqref{eq:thm-normalized-multipliers} and the continuity of $\grad f$, $\rmD g$, and $\rmD h$ yields
\begin{align}
\rmD g(x^{\ast})^{\ast}[\overline y] + \rmD h(x^{\ast})^{\ast}[\overline z]=0, \quad \langle\overline z,h(x^{\ast})\rangle=0.
\label{eq:thm-abnormal-stationarity}
\end{align}
Now, we define
\begin{align*}
{\cal T} \coloneqq
\begin{bmatrix}
g(x^{\ast})
\\
h(x^{\ast})
\end{bmatrix}
+
\begin{bmatrix}
\rmD g(x^{\ast})
\\
\rmD h(x^{\ast})
\end{bmatrix}
T_{x^{\ast}}\M
-
\begin{bmatrix}
\{0\}
\\
\R^{\ell}_{+}
\end{bmatrix}.
\end{align*}
The RCQ gives $0 \in \mathrm{int}({\cal T})$. Thus, by noting $(\overline{y},\overline{z}) \neq 0$, there exists a sufficiently small positive number $\epsilon > 0$ such that $(\epsilon \overline{y}, \epsilon \overline{z}) \in {\cal T}$. Let $p \coloneqq \epsilon \overline y$ and $q \coloneqq \epsilon \overline z$. Since $x^{\ast}$ is feasible for~\eqref{pro:RNLO}, the element $(p, q) \in {\cal T}$ can be written as $p = \rmD g(x^{\ast})[\xi]$ and $q = h(x^{\ast}) + \rmD h(x^{\ast})[\xi] - w$ for some $\xi\in T_{x^{\ast}}\M$ and $w\in\R^{\ell}_{+}$. It then follows from~\eqref{eq:thm-abnormal-stationarity} that 
\begin{align}
\langle \overline{y}, p \rangle + \langle \overline{z}, q \rangle 
= \langle \rmD g(x^{\ast})^{\ast} [\overline{y}] + \rmD h(x^{\ast})^{\ast} [\overline{z}], \xi \rangle_{x^{\ast}} + \langle \overline{z}, h(x^{\ast}) \rangle - \langle \overline{z}, w \rangle = - \langle \overline{z}, w \rangle \leq 0, \label{ineq:p_q}
\end{align}
where the inequality is derived from $\overline{z} \geq 0$ and $w \geq 0$. Meanwhile, the last equality of~\eqref{eq:thm-normalized-multipliers} and the definitions of $p$ and $q$ imply $\langle \overline{y}, p \rangle + \langle \overline{z}, q \rangle = \epsilon \| \overline{y} \|^2 + \epsilon \| \overline{z} \|^2 = \epsilon$. Thus, using~\eqref{ineq:p_q} leads to $\epsilon \leq 0$. However, this contradicts $\epsilon > 0$. Therefore, $\{(y_k, z_k)\}_{k \in {\cal Q}}$ is bounded.
\par
Consequently, there exist $\mathcal U \subset \mathcal Q$, $y^{\ast} \in \R^{m}$, and $z^{\ast} \in \R^{\ell}_{+}$ such that $y_k \to y^{\ast}$ and $z_k \to z^{\ast}$ as $\mathcal U \ni  k \to \infty$. By noting~\eqref{equation:gradL_zh}, the feasibility of $x^{\ast}$, and the continuity of $\grad f$, $\rmD g$, and $\rmD h$, we obtain
\begin{align*}
\grad_x L(x^{\ast},y^{\ast},z^{\ast}) = 0, \quad \langle z^{\ast},h(x^{\ast})\rangle = 0.
\end{align*}
Together with the feasibility of $x^{\ast}$ and $z^{\ast}\geq0$, these relations show that $(x^{\ast},y^{\ast},z^{\ast})$ satisfies the KKT conditions of~\eqref{pro:RNLO}.
\end{proof}

\section{Numerical experiments}\label{sec:Numerical_experiments}
This section compares Algorithm~\ref{alg:main} with the Riemannian sequential quadratic optimization method proposed by Obara, Okuno, and Takeda~\cite{Obara:2022}, which is referred to as the RSQP method in this paper. Both methods were implemented in \mbox{MATLAB} \mbox{R2025a} using the \mbox{Manopt} library. The experiments were performed on a computer equipped with an Intel Core i9-9900K CPU (3.60 GHz).
\par
For Algorithm~\ref{alg:main}, the algorithmic parameters were set to $\beta \coloneqq 0.5$, $\varepsilon \coloneqq 10^{-4}$, $\theta \coloneqq 10^{-4}$, $C \coloneqq \{ y \in \mathbb{R}^{m}; -10^6e \leq y \leq 10^6e \}$, and $D \coloneqq \{ z \in \mathbb{R}^{\ell}; 0 \leq z \leq 10^6 e \}$, where $e$ denotes a vector of ones of the appropriate dimension. We also set $y_0 \coloneqq 0$, $z_0 \coloneqq 0$, $\sigma_0 \coloneqq 1$, $\phi_0 \coloneqq 1$, $\psi_0 \coloneqq 1$, and $\gamma_0 \coloneqq 1$. The initial primal point $x_0$ is specified separately for each test problem below. For the RSQP method, we used the same initial primal point as that used in Algorithm~\ref{alg:main} and adopted the same algorithmic parameter settings as those employed in the numerical experiments of~\cite{Obara:2022}. For both Algorithm~\ref{alg:main} and the RSQP method, the operator $H_k$ was set to the Riemannian Hessian $\Hess_xL(x_k,y_k,z_k)$ by replacing every eigenvalue smaller
than $10^{-5}$ with $10^{-5}$. At each iteration, the quadratic subproblem in each method was solved using \texttt{quadprog}, a quadratic programming solver available in MATLAB. A run was terminated when either $r(x_k,y_k,z_k)\leq 10^{-6}$ or $k=500$. Thus, an iteration count of 500 indicates that the prescribed residual tolerance was not attained within the maximum number of iterations.
\par
First, we consider the following mathematical program with equilibrium constraints (MPEC) on the unit sphere, which is used in the numerical experiments of~\cite{Yamakawa:2022}:
\begin{align}
\begin{aligned}\label{pro:MPEC}
& \mini_{x \in \mathbb{S}^{n-1}}
&& \frac{1}{2} \langle Ax, x \rangle + \langle b, x \rangle
\\
& \subj
&& \langle Gx+p, Hx+q \rangle = 0,
\\
&&& Gx+p \geq 0,
\\
&&& Hx+q \geq 0,
\end{aligned}
\end{align}
where $\mathbb{S}^{n-1} \coloneqq \{x \in \R^n ; \|x\|_2=1\}$, $A \in \R^{n \times n}$, $G \in \R^{m \times n}$, $H \in \R^{m \times n}$, $b \in \R^n$, $p \in \R^m$, and $q \in \R^m$. The entries of these matrices and vectors were sampled from the uniform distribution on $[-1,1]$. Problem~\eqref{pro:MPEC} is degenerate in the sense that the RCQ fails at every feasible point.
\par
We considered the following five problem sizes:
\begin{align*}
(n,m) \in \{(50,15),(100,15),(100,25),(200,30),(200,45)\}.
\end{align*}
For each pair $(n,m)$, both methods were applied to the same generated problem instance. The initial primal point $x_0$ was generated by the command \verb|M.rand()| in the Manopt library, and the same point was used for both methods.
\par
Second, we consider the following nonnegative low-rank matrix completion problem used in the numerical experiments of~\cite{Obara:2022}:
\begin{align}
\begin{aligned}\label{pro:LRMC}
& \mini_{X\in\mathcal{M}_{p}}
&& \frac{1}{2}
\left\|
\mathcal P_{\mathcal J \setminus \mathcal C}(X-A)
\right\|_{F}^{2}
\\
& \subj
&& X_{ij} \geq 0
\quad
\forall (i,j) \in \mathcal S \backslash \mathcal J,
\\
&&& X_{ij}=A_{ij}
\quad
\forall (i,j) \in \mathcal C,
\end{aligned}
\end{align}
where $\mathcal{M}_{p} \coloneqq \{X \in \R^{q \times s};\ {\rm rank}(X) = p\}$ is the fixed-rank manifold and $\mathcal S \coloneqq \{ 1, \ldots, q \} \times \{ 1, \ldots, s \}$, $\mathcal J \subset \mathcal S$, $\mathcal C \subset \mathcal J$, and $A \in \R^{q \times s}$ are given. The entries of $A$ indexed by $\mathcal J$ are observed, and those indexed by $\mathcal C$ are treated as exact, whereas the entries indexed by $\mathcal J \backslash \mathcal C$ are used in the least-squares objective. For $Z \in \R^{q \times s}$, $\mathcal P_{\mathcal J \backslash \mathcal C}(Z)$ denotes the matrix whose $(i,j)$-entry is $Z_{ij}$ if $(i,j) \in \mathcal J \backslash \mathcal C$ and zero otherwise. Strictly speaking, $\mathcal{M}_{p}$ is not complete and therefore does not satisfy the completeness assumption imposed in the theoretical analysis of this study. Nevertheless, as in~\cite{Obara:2022}, we include problem~\eqref{pro:LRMC} to examine the practical performance of the methods on a representative fixed-rank matrix problem.
\par
As described in~\cite{Obara:2022}, for each triple $(q,s,p)$, we first generated matrices $T \in \R^{q \times p}$ and $V \in \R^{p\times s}$ whose entries were sampled from the uniform distribution on $[0,1]$. This procedure was repeated until ${\rm rank}(TV)=p$, and we then set $A\coloneqq TV$. The sets $\mathcal J$ and $\mathcal C$ were randomly generated so that $|\mathcal J|=\lceil |\mathcal S|/2\rceil$ and $|\mathcal C|=\lceil |\mathcal J|/2\rceil$. We considered the following five problem sizes:
\begin{align*}
(q,s,p)
\in
\{
(5,10,2),
(10,30,3),
(15,30,4),
(20,25,4),
(20,45,7)
\}.
\end{align*}
For each triple $(q,s,p)$, both methods were applied to the same generated problem instance.
\par
According to~\cite{Obara:2022}, we used an approximate solution of the feasibility problem associated with problem~\eqref{pro:LRMC}, with constraint violation no greater than $10^{-2}$, as the common initial point. To obtain such a point, we applied the \verb|conjugategradient| solver in the Manopt library to the following unconstrained problem:
\begin{align*}
\mini_{X\in\mathcal{M}_{p}}
\quad
\frac{1}{2}
\sum_{(i,j)\in \mathcal C}(X_{ij}-A_{ij})^2
+
\frac{1}{2}
\sum_{(i,j)\in \mathcal S \backslash \mathcal J}[-X_{ij}]_+^2.
\end{align*}
Since $A\in\mathcal{M}_{p}$ is componentwise nonnegative, the optimal value of this problem is zero, and its global minimizers coincide exactly with the feasible set of the feasibility problem. Starting from a point generated by \verb|M.rand()|, the computation was restarted from a new random point whenever the square root of twice the resulting objective value was greater than $10^{-2}$. Only the entries of $A$ indexed by $\mathcal C$ were used in this initialization procedure, and the same accepted point was used as $x_0$ for both methods.
\par
The numerical results for problems~\eqref{pro:MPEC} and~\eqref{pro:LRMC} are summarized in Tables~\ref{tb:RSSQP}--\ref{tb:RSQP2}. For problem~\eqref{pro:MPEC}, Algorithm~\ref{alg:main} attained the prescribed tolerance for all five test instances, whereas the RSQP method reached the maximum number of iterations in every case without attaining the tolerance. These results support the effectiveness of the proposed stabilization when the RCQ fails. For problem~\eqref{pro:LRMC}, both methods attained the prescribed tolerance for all five instances, showing that Algorithm~\ref{alg:main} retains performance comparable to the RSQP method on instances for which the latter also works successfully. Moreover, Algorithm~\ref{alg:main} required fewer iterations and less CPU time for the two larger instances. This tendency may be partly attributed to the fact that the stabilized subproblem is always feasible regardless of the infeasibility of the current iteration point, whereas the subproblem of the RSQP method does not enjoy such a guarantee.

\begin{table}[ht]
\centering
\caption{Numerical results for Algorithm~\ref{alg:main} on problem~\eqref{pro:MPEC}.}
\vspace{-2.5mm}
\begin{tabular}{|c||c|c|c|c|c|} \hline
$(n,m)$ & (50, 15) & (100, 15) & (100, 25) & (200, 30) & (200, 45)
\\ \hline
Iterations & 64 & 84 & 64 & 68 & 77
\\
CPU time (s) & 4.63 & 9.83 & 15.20 & 43.58 & 70.53
\\
$r(x_k, y_k, z_k)$ & 7.93e$-$07 & 9.21e$-$07 & 8.61e$-$07 & 9.80e$-$07 & 9.46e$-$07
\\
\hline
\end{tabular}
\label{tb:RSSQP}
\vspace{5mm}
\centering
\caption{Numerical results for the RSQP method on problem~\eqref{pro:MPEC}.}
\vspace{-2.5mm}
\begin{tabular}{|c||c|c|c|c|c|} \hline
$(n,m)$ & (50, 15) & (100, 15) & (100, 25) & (200, 30) & (200, 45)
\\ \hline
Iterations & 500 & 500 & 500 & 500 & 500
\\
CPU time (s) & 5.99 & 27.32 & 64.84 & 57.15 & 249.02
\\
$r(x_k, y_k, z_k)$ & 9.06e$-$02 & 2.90e$-$01 & 5.19e$-$01 & 2.57e$-$01 & 1.83e$+$00
\\
\hline
\end{tabular}
\label{tb:RSQP}
\vspace{5mm}
\centering
\caption{Numerical results for Algorithm~\ref{alg:main} on problem~\eqref{pro:LRMC}.}
\vspace{-2.5mm}
\begin{tabular}{|c||c|c|c|c|c|} \hline
$(q,s,p)$ & (5, 10, 2) & (10, 30, 3) & (15, 30, 4) & (20, 25, 4) & (20, 45, 7)
\\ \hline
Iterations & 19 & 35 & 79 & 86 & 74
\\
CPU time (s) & 0.47 & 12.20 & 66.08 & 80.72 & 314.42
\\
$r(x_k, y_k, z_k)$ & 4.52e$-$07 & 2.11e$-$07 & 9.22e$-$07 & 8.68e$-$07 & 9.66e$-$07
\\
\hline
\end{tabular}
\label{tb:RSSQP2}
\vspace{5mm}
\centering
\caption{Numerical results for the RSQP method on problem~\eqref{pro:LRMC}.}
\vspace{-2.5mm}
\begin{tabular}{|c||c|c|c|c|c|} \hline
$(q,s,p)$ & (5, 10, 2) & (10, 30, 3) & (15, 30, 4) & (20, 25, 4) & (20, 45, 7)
\\ \hline
Iterations & 9 & 35 & 67 & 141 & 135
\\
CPU time (s) & 0.21 & 12.13 & 54.67 & 129.19 & 572.35
\\
$r(x_k, y_k, z_k)$ & 4.85e$-$07 & 5.82e$-$07 & 6.18e$-$07 & 9.66e$-$07 & 9.27e$-$07
\\
\hline
\end{tabular}
\label{tb:RSQP2}
\end{table}

\section{Conclusion}\label{sec:conclusion}
In this paper, we proposed a stabilized SQP method for Riemannian nonlinear programming problems with equality and inequality constraints. The method combines stabilized quadratic subproblems, an augmented-Lagrangian merit function, and update rules designed to control the multiplier estimates. To the best of the authors' knowledge, this is the first stabilized SQP framework developed for constrained optimization on Riemannian manifolds. The convergence analysis establishes the existence of an accumulation point that is a KKT point, an AKKT point, or a stationary point of the associated feasibility problem. If the accumulation point satisfies the ERCQ, then it is a KKT point. Thus, the method provides a meaningful global convergence characterization even when no constraint qualification is imposed. The numerical results show that the proposed RSSQP method is clearly more effective than the existing RSQP method for all the tested degenerate MPEC instances on the unit sphere. Moreover, for the nonnegative low-rank matrix completion problems, the RSSQP method remained competitive with the RSQP method and outperformed it on the larger instances.
\par
An important direction for future research is to establish local superlinear and quadratic convergence of the proposed RSSQP method under suitable second-order and regularity assumptions.

\addcontentsline{toc}{section}{\refname} 

\bibliographystyle{abbrv}
\bibliography{refs}

\appendix

\section{Proofs} \label{sec:appendix}
The proof of Proposition~\ref{prop:solution} is as follows.
\begin{proof}
For simplicity, we omit the subscript $k$ and define ${\cal F} \colon T_{x} \M \times \R^{\ell} \to \R$ and ${\cal S} \subset T_{x} \M \times \R^{\ell}$ as follows:
\begin{align}
&{\cal F}(\xi,\eta) \coloneqq \langle \grad f(x) - \rmD g(x)^{\ast}[s], \xi \rangle_x + \frac{1}{2} \langle M[\xi], \xi \rangle_{x} + \frac{\sigma}{2} \Vert \eta \Vert^{2},
\\
&{\cal S} \coloneqq \{(\xi, \eta) \in T_{x} \M \times \R^{\ell} ; \rmD h(x)[\xi] + \sigma (\eta - t) \geq 0 \}.
\end{align}
Since $M$ is coercive, there exists $\ell_M > 0$ such that $\langle M [v], v \rangle_{x} \geq \ell_M \Vert v \Vert_{x}^2$ for any $v \in T_{x} \M$. Hence, we have
\begin{align}
\begin{aligned}\label{eq:Fbound}
{\cal F}(\xi,\eta) 
& \geq \frac{\ell_{M}}{2} \Vert \xi \Vert_{x}^2 + \frac{\sigma}{2} \Vert \eta \Vert^{2} - \Vert \grad f(x) - \rmD g(x)^{\ast}[s] \Vert_{x} \Vert \xi \Vert_{x}
\\
& = \frac{\ell_{M}}{2} \left(\Vert \xi \Vert_{x} - \frac{1}{\ell_M} \Vert \grad f(x)-\rmD g(x)^{\ast}[s] \Vert_{x} \right)^{2}
\\
&\hspace{18mm} + \frac{\sigma}{2} \Vert \eta \Vert^{2} - \frac{1}{2 \ell_M} \Vert \grad f(x)-\rmD g(x)^{\ast}[s] \Vert_{x}^{2}
\\
& \geq -\frac{1}{2\ell_M} \Vert \grad f(x)-\rmD g(x)^{\ast}[s] \Vert_{x}^{2}.
\end{aligned}
\end{align}
Thus, the function ${\cal F}$ is bounded below and coercive. Moreover, it is clear that $(0, t) \in {\cal S}$, that is, ${\cal S} \not= \emptyset$. These facts imply $-\infty < \inf \{ {\cal F}(\xi, \eta); (\xi, \eta) \in {\cal S} \} < \infty$. Then, there exists $\{ (\xi_j, \eta_j) \} \subset {\cal S}$ such that 
\begin{align}
\lim_{j \to \infty} {\cal F}(\xi_j,\eta_j) = \inf \{ {\cal F}(\xi,\eta); (\xi,\eta)\in {\cal S} \}, \label{eq:Fxieta}
\end{align}
which means that $\{ {\cal F}(\xi_j,\eta_j) \}$ is bounded. Since ${\cal F}$ is coercive, this fact implies that $\{(\xi_j,\eta_j)\}$ is also bounded. Hence, there exists a convergent subsequence $\{ (\xi_j, \eta_j) \}_{j \in {\cal P}} \subset \{ (\xi_j, \eta_j) \}$ such that $(\xi_j, \eta_j) \to (\widehat{\xi},\widehat{\eta})$ as $j \to \infty$, $j \in {\cal P}$. Meanwhile, because $\{ (\xi_j, \eta_j) \}_{j \in {\cal P}} \subset {\cal S}$ holds and ${\cal S}$ is closed, we have $(\widehat{\xi},\widehat{\eta}) \in {\cal S}$. It then follows from~\eqref{eq:Fxieta} that ${\cal F}(\widehat{\xi}, \widehat{\eta}) = \inf \{ {\cal F}(\xi, \eta); (\xi, \eta) \in {\cal S} \}$. Therefore, the solvability of~\eqref{pro:RNLOparch} is verified.
\par
Next, we show the latter assertion. To this end, we verify that each feasible point of~\eqref{pro:RNLOparch} satisfies the RCQ. Notice that the mapping $T_x \M \times \R^{\ell} \ni (\xi, \eta) \mapsto \rmD h(x)[\xi] + \sigma \eta \in \R^{\ell}$ is surjective, namely, the image of $T_x \M \times \R^{\ell}$ under this mapping is equal to $\R^{\ell}$. The fact ensures that each feasible point of~\eqref{pro:RNLOparch} meets the RCQ. Therefore, the KKT conditions are satisfied at $(\widehat{\xi}, \widehat{\eta})$.
\par
Since $M$ is self-adjoint and coercive and $\sigma>0$, the objective function $\mathcal F$ is strongly convex. Moreover, the feasible set $\mathcal S$ is convex. Therefore, problem~\eqref{pro:RNLOparch} has a unique optimal solution.
\end{proof}

\noindent
The proof of Lemma~\ref{lem:grad} is as follows.
\begin{proof}
For simplicity, we omit the subscript $k$. According to Proposition~\ref{prop:solution}, Problem~\eqref{pro:RNLOparch} has a unique optimal solution~$(\xi^{\ast}, \eta^{\ast})$ and it satisfies the KKT conditions of problem~\eqref{pro:RNLOparch}, that is,
\begin{align}
&M [\xi^{\ast}] + \grad f(x) - \rmD g(x)^{\ast}[s]- \rmD h(x)^{\ast} [\eta^{\ast}]=0, \label{eq:parchKKT1} 
\\
& \rmD h(x)[\xi^{\ast}] + \sigma(\eta^{\ast} - t) \geq 0,\label{eq:parchKKT2}
\\
& \eta^{\ast} \geq 0,\label{eq:parchKKT3}
\\
& \langle \eta^{\ast}, \rmD h(x) [\xi^{\ast}] + \sigma(\eta^{\ast} - t) \rangle=0. \label{eq:parchKKT4}
\end{align}
Noting~\eqref{eq:graddef} leads to
\begin{align*}
\langle  \grad F(x; \widehat{y}, \widehat{z}, \sigma), \xi^{\ast} \rangle_{x}
= \langle \grad f(x) - \rmD g(x)^{\ast}[s], \xi^{\ast} \rangle_{x} - \langle \rmD h(x)^{\ast}[[t]_{+}], \xi^{\ast} \rangle_x.
\end{align*}
Using~\eqref{eq:parchKKT1} and~\eqref{eq:parchKKT4} yields
\begin{align}
\langle \grad F(x; \widehat{y}, \widehat{z}, \sigma), \xi^{\ast} \rangle_x
& = - \langle M[\xi^{\ast}], \xi^{\ast} \rangle_x + \langle \eta^{\ast}, \rmD h(x) [\xi^{\ast}] \rangle - \langle [t]_{+}, \rmD h(x)[\xi^{\ast}] \rangle \nonumber
\\
&= - \langle M[\xi^{\ast}], \xi^{\ast} \rangle_x - \sigma \langle \eta^{\ast}, \eta^{\ast} - t \rangle - \langle [t]_{+}, \rmD h(x)[\xi^{\ast}] \rangle. \label{ineq:gradF_M_xi}
\end{align}
Meanwhile, combining~\eqref{eq:parchKKT2} and $[t]_{+} \geq 0$ implies $0 \leq \langle [t]_{+}, \sigma ( \eta^{\ast} - t) + \rmD h(x) [\xi^{\ast}] \rangle$. Thus, we have
\begin{align}
- \langle [t]_{+}, \rmD h(x)[\xi^{\ast}] \rangle 
&\leq \sigma \langle [t]_{+}, \eta^{\ast} - t \rangle \nonumber
\\
&= \sigma \langle [t]_{+} - \eta^{\ast}, \eta^{\ast} - t \rangle + \sigma \langle \eta^{\ast}, \eta^{\ast} - t \rangle \nonumber
\\
&= - \sigma \Vert [t]_{+} - \eta^{\ast} \Vert^2 + \sigma \langle [t]_{+} - \eta^{\ast}, [t]_{+} - t \rangle + \sigma \langle \eta^{\ast}, \eta^{\ast} - t \rangle \nonumber
\\
&\leq - \sigma \Vert [t]_{+} - \eta^{\ast} \Vert^2 + \sigma \langle \eta^{\ast}, \eta^{\ast} - t \rangle, \label{ineq:gradF_t_eta}
\end{align}
where the last inequality is derived from~\eqref{eq:parchKKT3} and the well-known property of the projection $[ \, \cdot \, ]_{+}$. It follows from~\eqref{ineq:gradF_M_xi} and~\eqref{ineq:gradF_t_eta} that
\begin{align*}
\langle \grad F(x; \widehat{y}, \widehat{z}, \sigma), \xi^{\ast} \rangle_x \leq - \langle M[\xi^{\ast}], \xi^{\ast} \rangle_x - \sigma \| \eta^{\ast} - [t]_+ \|^{2},
\end{align*}
that is,~\eqref{eq:gradient} is verified. Moreover, by exploiting~\eqref{eq:graddef},~\eqref{eq:parchKKT1}, and the definitions of $s$ and $t$, we obtain
\begin{align*}
0 &= M[\xi^{\ast}] + \grad f(x) - \rmD g(x)^{\ast}[s] - \rmD h(x)^{\ast}[\eta^{\ast}]
\\
& = M[\xi^{\ast}] + \grad F(x; \widehat{y}, \widehat{z}, \sigma) - \rmD h(x)^{\ast} \left[ \eta^{\ast} - [t]_+ \right].
\end{align*}
This equality yields
\begin{align*}
\grad F(x; \widehat{y}, \widehat{z}, \sigma) = - M[\xi^{\ast}] + \rmD h(x)^{\ast} \left[ \eta^{\ast} - [t]_+ \right],
\end{align*}
and hence~\eqref{eq:ineqabs} is satisfied. If $\grad F(x; \widehat{y}, \widehat{z}, \sigma) = 0$ holds, we easily see that $\xi^{\ast} = 0$ and $\eta^{\ast} = [t]_+$ by~\eqref{eq:gradient} and the coercivity of $M$. Namely, it is clear that $(0, [t]_+)$ is a unique optimum. Conversely, if $(0, [t]_+)$ is a unique optimum, it follows from~\eqref{eq:graddef} and~\eqref{eq:parchKKT1} that $\grad F(x; \widehat{y}, \widehat{z}, \sigma) = 0$.
\end{proof}

\end{document}